\documentclass[11pt]{article}

\usepackage[margin=1in]{geometry}

\usepackage{lipsum}

\usepackage[utf8]{inputenc} 
\usepackage[T1]{fontenc}    
\usepackage{tgtermes}

\usepackage{url}            
\usepackage{booktabs}       
\usepackage{amsfonts}       
\usepackage{nicefrac}       
\usepackage{microtype}      

\usepackage[english]{babel}

\usepackage{multirow}
\usepackage{multicol}
\usepackage{amsmath}
\usepackage{enumitem}
\usepackage{amssymb}
\usepackage{graphicx}
\usepackage{bm}
\usepackage{theorem}
\usepackage[colorlinks=true, allcolors=blue]{hyperref}

\newcommand{\commHL}[1]{{\textcolor{blue}{#1}}}
\DeclareMathOperator*{\argmin}{arg\,min}
\def\R{\mathbb{R}}
\def\E{\mathbb{E}}

\usepackage{accents}
\newlength{\dhatheight}

\usepackage{mathtools}
\newenvironment{proof}{\paragraph{\it Proof.}}{\hfill$\square$}

\usepackage[title]{appendix}
\usepackage{comment}
\usepackage{amsfonts}
\usepackage{dsfont}
\usepackage{hyperref}
\usepackage{lipsum}
\usepackage{dsfont,subcaption,float}
\usepackage{algorithm}
\usepackage{algorithmic}
\usepackage{cleveref}

\usepackage{thm-restate}
\usepackage[size=small,labelfont=bf]{caption}

\usepackage{xcolor}

\theoremstyle{plain}
\newtheorem{theorem}{Theorem}[section]

\newtheorem{proposition}[theorem]{Proposition}
\newtheorem{lemma}[theorem]{Lemma}

\newtheorem{definition}[theorem]{Definition}
\newtheorem{assumption}{Assumption}

\newtheorem{conjecture}{Conjecture}

\newcommand{\supp}{\mathrm{supp}}

\newcommand{\mA}{\mathcal{A}}	
\newcommand{\mO}{\mathcal{O}}	
\newcommand{\mY}{\mathcal{Y}}	

\newcommand{\mX}{\mathcal{X}}	
\newcommand{\mF}{\mathcal{F}}	
\newcommand{\mL}{\mathcal{L}}	
\newcommand{\indic}[1]{\mathds{1}_{\left\{ #1 \right\}}} 

\newcommand{\PP}{\mathds{P}}    
\newcommand{\vdot}[2]{\langle #1, #2 \rangle}
\newcommand{\floor}[1]{\lfloor #1 \rfloor}

\newcommand{\Proj}[2]{\Pi_{ #1 } \left\{ #2 \right\}}

\newcommand{\Var}{\normalfont\texttt{Var}}

\newcommand{\Exs}{\mathbb{E}}
\newcommand{\prog}{\normalfont\texttt{prog}}
\newcommand{\Ber}{\texttt{Ber}}

\newcommand{\tssum}{\textstyle \sum}

\newcommand{\grad}{\nabla}

\allowdisplaybreaks

\newif\ifdraft
\drafttrue  

\newif\ifarxiv

\newcommand{\blue}[1]{\textcolor{blue}{#1}}

\newcommand{\dkcomment}[1]{\ifdraft {\bf{{\blue{{Dohyun --- #1}}}}}\else\fi}

\newcommand{\I}{\mathcal{I}}

\newcommand{\W}{\mathcal{W}}

\arxivtrue

\title{\bf{\LARGE{On the Complexity of First-Order Methods in Stochastic Bilevel Optimization}}}

\usepackage{authblk}
\author[1]{Jeongyeol Kwon}
\author[2]{Dohyun Kwon}
\author[1]{Hanbaek Lyu}
\affil[1]{University of Wisconsin-Madison}
\affil[2]{University of Seoul}

\begin{document}
\maketitle

\begin{abstract}
    We consider the problem of finding stationary points in Bilevel optimization when the lower-level problem is unconstrained and strongly convex. The problem has been extensively studied in recent years; the main technical challenge is to keep track of lower-level solutions $y^*(x)$ in response to the changes in the upper-level variables $x$. Subsequently, all existing approaches tie their analyses to a genie algorithm that knows lower-level solutions and, therefore, need not query any points far from them. We consider a dual question to such approaches: suppose we have an oracle, which we call $y^*$-aware, that returns an $O(\epsilon)$-estimate of the lower-level solution, in addition to first-order gradient estimators {\it locally unbiased} within the $\Theta(\epsilon)$-ball around $y^*(x)$. We study the complexity of finding stationary points with such an $y^*$-aware oracle: we propose a simple first-order method that converges to an $\epsilon$ stationary point using $O(\epsilon^{-6}), O(\epsilon^{-4})$ access to first-order $y^*$-aware oracles. Our upper bounds also apply to standard unbiased first-order oracles, improving the best-known complexity of first-order methods by $O(\epsilon)$ with minimal assumptions. We then provide the matching $\Omega(\epsilon^{-6})$, $\Omega(\epsilon^{-4})$ lower bounds without and with an additional smoothness assumption on $y^*$-aware oracles, respectively. Our results imply that any approach that simulates an algorithm with an $y^*$-aware oracle must suffer the same lower bounds.

\end{abstract}

\section{Introduction}
\label{section:intro}

Bilevel optimization \cite{colson2007overview} is a fundamental optimization problem that abstracts the core of various critical applications characterized by two-level hierarchical structures, including meta-learning \cite{rajeswaran2019meta}, hyper-parameter optimization \cite{franceschi2018bilevel, bao2021stability}, model selection \cite{kunapuli2008bilevel, giovannelli2021bilevel}, adversarial networks \cite{goodfellow2020generative, gidel2018variational}, game theory \cite{stackelberg1952theory} and reinforcement learning \cite{konda1999actor, sutton2018reinforcement}. In essence, Bilevel optimization can be abstractly described as the subsequent minimization problem:
\begin{align}
    &\min_{x \in \mathbb{R}^{d_x}} \quad F(x) := f(x,y^*(x)) \nonumber \\
    &\text{s.t.} \quad y^*(x) \in \arg \min_{y \in \mathbb{R}^{d_y}} g(x,y), \label{problem:bilevel} \tag{\textbf{P}}
\end{align}
where $f,g: \mathbb{R}^{d_x} \times \mathbb{R}^{d_y} \rightarrow \mathbb{R}$ are continuously-differentiable functions. The \textit{hyperobjective} $F(x)$ depends on $x$ both directly and indirectly via $y^*(x)$, which is a solution for the lower-level problem of minimizing another function $g$, which is parametrized by $x$. Throughout the paper, we assume that the lower-level problem is strongly-convex, {\it i.e.,} $g(\bar{x},y)$ is strongly convex in $y$ for all $\bar{x} \in \mathbb{R}^{d_x}$. 

Our goal is to find an \textit{$\epsilon$-stationary point} of \eqref{problem:bilevel}: an $x$ that satisfies $\|\grad F(x) \| \le \epsilon$. Here the explicit expression of $\grad F(x)$ can be derived from the implicit function theorem \cite{krantz2002implicit}:
\ifarxiv
\begin{align}
    \grad F(x) := \grad f(x,y^*(x)) - \grad_{xy}^2 g(x,y^*(x)) \left(\grad_{yy}^2 g(x,y^*(x))\right)^{-1} \grad_y f(x,y^*(x)).
    \label{eq:explicit_grad_F}
\end{align}
\else
\begin{align} \label{eq:explicit_grad_F}
    \grad F(x) &= \grad_x f (x, y^*(x)) 
      \\
    &\hspace{-0.8cm} - \grad_{xy}^2 g(x, y^*(x))   \grad_{yy}^2 g(x, y^*(x))^{-1} \grad_y f(x, y^*(x)). \nonumber
\end{align}
\fi
Following the standard black-box optimization model \cite{nemirovskij1983problem}, we consider the first-order algorithm class that accesses functions through \textit{first-order oracles} that return estimators of first-order derivatives $\hat{\nabla} f(x,y; \zeta), \hat{\nabla} g(x,y; \xi)$ for a given query point $(x,y)$ such that the following holds:
\ifarxiv
\begin{align}
    &\hspace{-0.2cm}\Exs[\hat{\nabla} f(x,y; \zeta) ] = \grad f(x,y),\ \Exs[\hat{\grad} g(x,y;\xi)] = \grad g(x,y), \label{eq:unbiased_gradients} \\
    &\hspace{-0.2cm}\Exs[\| \hat{\nabla} f(x,y; \zeta) - \Exs[\nabla f(x, y; \zeta)] \|^2] \le \sigma_f^2, \nonumber \\
    &\hspace{-0.2cm}\Exs[\|\hat{\nabla} g(x,y; \xi) - \Exs[\nabla g(x, y; \xi)]\|^2] \le \sigma_g^2, \label{eq:bounded_variance_gradients}  \\
  & \hspace{-0.2cm} \Exs[\| \hat{\nabla} g(x,y^1; \xi) - \hat{\nabla} g(x, y^2; \xi)\|^2] \le \tilde{l}_{g,1} \|y^1 - y^2\|^2, \label{eq:mean_squared_lipschitz} 
\end{align}  
\else
\begin{align}
    &\hspace{-0.2cm}\Exs[\hat{\nabla} f(x,y; \zeta) ] = \grad f(x,y), \nonumber \\
    &\hspace{-0.2cm} \Exs[\hat{\grad} g(x,y;\xi)] = \grad g(x,y), \label{eq:unbiased_gradients} \\
    &\hspace{-0.2cm}\Exs[\| \hat{\nabla} f(x,y; \zeta) - \Exs[\nabla f(x, y; \zeta)] \|^2] \le \sigma_f^2, \nonumber \\
    &\hspace{-0.2cm}\Exs[\|\hat{\nabla} g(x,y; \xi) - \Exs[\nabla g(x, y; \xi)]\|^2] \le \sigma_g^2, \label{eq:bounded_variance_gradients}  \\
  & \hspace{-0.2cm} \Exs[\| \hat{\nabla} g(x,y^1; \xi) - \hat{\nabla} g(x, y^2; \xi)\|^2] \le \tilde{l}_{g,1} \|y^1 - y^2\|^2, \label{eq:mean_squared_lipschitz} 
\end{align}  
\fi
where $\sigma_f^2, \sigma_g^2 >0$ are the variance of gradient estimators, $\tilde{l}_{g,1}\in (0,\infty]$ is the stochastic smoothness parameter, and $\zeta, \xi$ are independently sampled random variables. The complexity of an algorithm is measured by the worst-case expected number of calls for the first-order oracles until an $\epsilon$-stationary point of \eqref{problem:bilevel} is found.

\paragraph{$y^*(x)$-Aware Oracles.} As we can see in the expression of hypergradients \eqref{eq:explicit_grad_F}, two main challenges distinguish (stochastic) Bilevel optimization from the more standard single-level optimization: at every $k^{th}$ iteration at a query point $x^k$, we need to estimate \textbf{(i)} $y^*(x^k)$, and then estimate \textbf{(ii)} $\grad F(x^k)$, which involves the estimation of Hessian-inverse using $y^k$. The vast volume of literature has been dedicated to resolving these two issues, starting from double-loop implementations which wait until $y^k$ to be sufficiently close to $y^*(x^k)$ \cite{ghadimi2018approximation}, to recent fully single-loop approaches that updates all variables within $O(1)$-oracle access with incremental improvement of $y^*(x^k)$ estimators \cite{dagreou2022framework, yang2023achieving}. Notably, all existing approaches require reasonable estimators of $y^{*}(x)$ for designing an algorithm for Bilevel optimization.

From a practical perspective, while it is standard in Bilevel literature to assume that the inner objective $g$ is strongly convex everywhere, many problems are globally nonconvex, and only \textit{locally strongly convex near $y^{*}(x)$}. In such cases, a common practice is to obtain a good initialization of $y^*(x)$ via more computationally expensive methods before running faster iterative algorithms to obtain more accurate estimates of lower-level solutions \cite{jain2010guaranteed, kwon2020algorithm, han2022optimal}. 


From these motivations, we formulate 
the following mathematical question: if we can directly obtain a sufficiently good estimator $\hat{y}(x)$ of $y^*(x)$ without additional complexity ({\it e.g.,} an oracle additionally provides an $\epsilon$-accurate estimate of $y^*(x)$ for a given $x$), and if we can only query gradients at points $(x,y)$ near $(x,y^{*}(x))$, what is the fundamental complexity of Bilevel problems? 
Formally, we consider the following oracle model which we refer to \textit{$y^*$-aware} oracle:

\vspace{0.1cm}
\begin{definition}[$y^*$-Aware Oracle]\label{def:ystart_aware}
     An oracle $\mathrm{O}(\cdot)$ is \textit{$y^*$-aware}, if there exists $r \in (0,\infty]$ such that for every query point $(x,y)$, the following conditions hold. \textup{\textbf{(i)}} in addition to stochastic gradients, the oracle also returns $\hat{y}(x)$ such that $\|\hat{y} (x) - y^*(x)\| \le r / 2$; \textup{\textbf{(ii)}} Gradient estimators satisfy \eqref{eq:unbiased_gradients}, \eqref{eq:bounded_variance_gradients} and \eqref{eq:mean_squared_lipschitz} only if $\|y - y^*(x)\| \le r$; otherwise, the returned gradient estimators can be arbitrary. 
\end{definition}
Note that, if we take $r=\infty$ in the above definition, then we recover the usual first-order stochastic gradient oracle that can be queried at any $(x,y)$ with the additional a priori estimator $\hat{y}(x)$ being uninformative. Thus, our oracle model subsumes the models conventionally assumed.

Conceptually, the complexity of any algorithm paired with an $y^*$-aware oracle can be considered as the lower limit of the problem, unless we can extract significant information from an arbitrary point $(x,y)$ where $y$ is far away from $y^*$. However, existing approaches view the information obtained at $y$ as meaningful only for the purpose of reaching $y^*(x)$, otherwise containing non-informative biases of size $O(\|y^*-y\|)$. For such approaches, one would expect faster convergence if sufficiently accurate estimates of $y^*(x)$ are provided for free. In this paper, we study lower bounds with such $y^*(x)$-aware oracles, providing a partial answer to the fundamental limits of the problem.

\paragraph{Prior Art.} Recent years have witnessed a rapid development of a body of work studying non-asymptotic convergence rates of iterative algorithms to $\epsilon$-stationary points of \eqref{problem:bilevel} under various assumptions on the stochastic oracles (see Section \ref{subsec:related_work} for the detailed overview). A major portion of existing literature assumes access to second-order information of $g$ via Jacobian/Hessian-vector product oracles (which we call  \textit{second-order oracles}) as the stationarity measure 
$\|\grad F(x)\|$ naturally requires computation of these quantities; see \eqref{eq:explicit_grad_F}. The best-known complexity results with second-order oracles give an $O(\epsilon^{-4})$ upper bound \cite{ji2021bilevel, chen2021closing}, and it can be improved to $O(\epsilon^{-3})$ with variance-reduction when the oracles have additional stochastic smoothness assumptions \cite{khanduri2021near, dagreou2022framework}.

A few recent works have shown that $\epsilon$-stationarity can also be achieved only with first-order oracles \cite{kwon2023fully, chen2023near, chen2023bilevel, lu2023first, yang2023achieving}. With stochastic noises, \cite{kwon2023fully} proposes an algorithm that finds $\epsilon$-stationary point within $O(\epsilon^{-7})$ access to first-order oracles, and $O(\epsilon^{-5})$ when the gradient estimators are (mean-squared) Lipschitz in expectation. Very recently, \cite{yang2023achieving} has shown that an $O(\epsilon^{-3})$ upper-bound is possible only with first-order oracles if we further have an additional {\it second-order} stochastic smoothness This result matches the best-known rate achieved by second-order baselines under the same condition. We close the remaining gap between first-order methods and second-order methods when we have fewer assumptions on stochastic oracles, {\it i.e.,} with the standard unbiased, variance-bounded, and possibly (mean-squared) gradient-Lipschitz oracles.

\subsection{Overview of Main Results}
We first provide high-level ideas on the expected convergence rates of first-order methods. To begin with, suppose for every $x$, we can directly access $y^*(x)$ and estimators of Jacobian/Hessian of $g$ without any cost. Then the problem becomes equivalent to finding a stationary point of $F(x)$ with (unbiased) estimators of $\grad F(x)$. After this reduction to single-level optimization, from the rich literature of nonconvex stochastic optimization \cite{arjevani2023lower}, the best achievable complexities are given as $\Theta(\epsilon^{-4})$ and $\Theta(\epsilon^{-3})$ without and with stochastic smoothness, respectively.

When we only have first-order oracles, the 
complexity can be similarly inferred from the previous bounds and simulating second-order oracles using first-order oracles. Indeed, observe that first-order oracles can simulate Jacobian/Hessian-vector product oracles (with a vector $v$) through finite differentiation with a precision parameter $\delta > 0$. Namely, $\Exs \left[\hat{\grad}_{xy}^2 g(x, y^*; \xi)^\top v \right]$ can be estimated as
\begin{align*}
     \Exs \left[\frac{\hat{\grad}_x g(x, y^*+\delta v; \xi) - \hat{\grad}_x g(x, y^*; \xi)}{\delta}\right] + O(\delta \|v\|^2). 
\end{align*}
However, without further assumptions, estimators of Jacobian-vector products obtained via the finite differentiation have $O(\delta)$-bias and $O(\delta^{-2})$-amplified variance. Thus, we can use $\delta$ at most $O(\epsilon)$ to keep the bias less than $\epsilon$ (hence the ``oracle-reliability radius'' should also satisfy $r = \Omega(\epsilon)$), and require $\Omega(\epsilon^{-2})$ times more oracle access to cancel out the variance amplification to approximately simulate the second-order methods, resulting in total $O(\epsilon^{-6})$ iterations. 


When we have stochastic smoothness, {\it i.e.,} gradient estimators are (mean-squared) Lipschitz as in \eqref{eq:mean_squared_lipschitz} (with $\tilde{l}_{g,1}<\infty$), then variances of finite-differentiation estimators are still bounded by $O(1)$, and we can obtain the $O(\epsilon^{-4})$ upper bound as if we can access second-order oracles.



\vspace{-0.2cm}
\paragraph{Upper Bound.} Given the above discussion, we derive upper bounds of $O(\epsilon^{-6})$, $O(\epsilon^{-4})$, without ($\tilde{l}_{g,1}=\infty$) and with ($\tilde{l}_{g,1}<\infty$ in \eqref{eq:mean_squared_lipschitz}) stochastic smoothness, respectively, with the $y^*$-aware oracle as long as $r = \Omega(\epsilon)$. In particular, with $r = \infty$, our results improve the best-known upper bounds given in \cite{kwon2023fully} by the order of $O(\epsilon)$ with minimal assumptions on stochastic oracles. Furthermore, perhaps surprisingly, our result shows that first-order methods are not necessarily worse than second-order methods under nearly the same assumption (as the stochastic smoothness assumption allows us to simulate second-order oracles with no additional cost in $\epsilon$). 

\vspace{-0.2cm}
\paragraph{Lower Bound.} Next, we turn our focus to lower bounds for finding an $\epsilon$-stationary point. For the lower bound, we assume $y^{*}(x)$-aware oracles with sufficiently good accuracy $r=\Theta(\epsilon)$ and we do not pursue lower bounds for the $r \gg \epsilon$ case in this work (see Conjecture \ref{conjecture:lower_bound}).
We prove that any black-box algorithms with $y^*(x)$-aware oracles must suffer at least $\Omega(\epsilon^{-6}), \Omega(\epsilon^{-4})$ access to oracles without and with stochastic smoothness, respectively. As an implication, if an algorithm (or analysis), including ours, does not introduce a slowdown by projecting $y$-coordinates of all query points onto a $\Theta(\epsilon)$-ball around $y^*(x)$, then its iteration complexity cannot be less than the proposed lower bounds.


\subsection{Our Approach}
Here, we overview our approaches for deriving the claimed upper and lower bounds.

\subsubsection{Upper Bound: Penalty Method}
In proving the upper bounds, our starting point is to consider an alternative reformulation of \eqref{problem:bilevel} through the penalty method used in \cite{kwon2023fully}. Specifically, consider a function $\mL_{\lambda}$ with a penalty parameter $\lambda > 0$:
\begin{align*}
    \mL_\lambda(x,y) := f(x,y) + \lambda (g(x,y) - g (x, y^*(x))).
\end{align*}
In \cite{kwon2023fully}, it was shown  that  the hyperobjective can be approximated by the following surrogate
\begin{align}
    \mL_{\lambda}^*(x) := \arg \min_y \mL_{\lambda}(x,y) \label{eq:penalty_surrogate}
\end{align}
in the sense that $\|\grad F(x) - \grad \mL_{\lambda}^*(x)\| \le O(1/\lambda)$, 
where $\grad \mL_{\lambda}^*(x)$ is given by
\ifarxiv
\begin{align}
    \grad \mL_{\lambda}^*(x) &= \grad_x f(x,y_{\lambda}^*(x)) + \lambda (\grad_x g(x,y_{\lambda}^*(x)) - \grad_x g(x,y^*(x))), \label{eq:grad_penalty}
\end{align}
\else
\begin{align}
    \grad \mL_{\lambda}^*(x) &= \grad_x f(x,y_{\lambda}^*(x)) \ + \nonumber \\
    &\quad \lambda (\grad_x g(x,y_{\lambda}^*(x)) - \grad_x g(x,y^*(x))), \label{eq:grad_penalty}
\end{align}
\fi
with $y_{\lambda}^*(x) := \arg \min_y \left(\lambda^{-1} f(x,y) + g(x,y)\right)$. Therefore, we can instead find an $\epsilon$-stationary point of $\mL_{\lambda}^*(x)$, {\it e.g.,} by running a stochastic gradient descent (SGD) style method on $\mL_{\lambda}^*(x)$ with $\lambda = O(\epsilon^{-1})$. 

For now, suppose $y_{\lambda}^*(x), y^*(x)$ are immediately accessible once $x$ is given. Then, we can construct the unbiased estimator $\hat{\grad} \mL_{\lambda}^*(x)$ with first-order oracles:
\ifarxiv
\begin{align*}
    \hat{\grad} \mL_{\lambda}^*(x) &= \hat{\grad} f(x,y_{\lambda}^*(x); \zeta) + \lambda(\hat{\grad} g(x, y_{\lambda}^*(x); \xi^y) - \hat{\grad} g(x,y^*(x); \xi^z)).
\end{align*}
\else
\begin{align*}
    \hat{\grad} \mL_{\lambda}^*(x) &= \hat{\grad} f(x,y_{\lambda}^*(x); \zeta) \ + \\
    &\quad \lambda(\hat{\grad} g(x, y_{\lambda}^*(x); \xi^y) - \hat{\grad} g(x,y^*(x); \xi^z)).
\end{align*}
\fi

\vspace{-0.2cm}
However, with the above construction, the variance of $\hat{\nabla} \mL_{\lambda}^*(x)$ is amplified by 
$\lambda^{2}$ so we need $O(\lambda^2) = O(\epsilon^{-2})$ batch of samples at every iteration to cancel out the amplified variance (similarly to finite-differentation for simulating Jacobian-vector products). Together with the standard sample complexity of SGD (which is $O(\texttt{Var}(\hat{\grad} \mL_{\lambda}^*) \epsilon^{-4})$), in total $O(\epsilon^{-6})$ oracle access 
should be sufficient. 

When we have stochastic smoothness \eqref{eq:mean_squared_lipschitz}, and if the oracle allows two query points for the same randomness (see Section \ref{section:prelim}) we can keep the variance controlled by coupling $y_{\lambda}^*$ and $y^*$ with the same random variable $\xi^y = \xi^z = \xi$ as
\ifarxiv
\begin{align*}
    \Var &\left(\hat{\grad}_x g(x, y_{\lambda}^*(x);\xi) - \hat{\grad}_x g(x,y^*(x);\xi)\right) \le \tilde{l}_{g,1}^2 \|y_{\lambda}^*(x) - y^*(x)\|^2.
\end{align*}
\else
\begin{align*}
    \Var &\left(\hat{\grad}_x g(x, y_{\lambda}^*(x);\xi) - \hat{\grad}_x g(x,y^*(x);\xi)\right) \\
    &\qquad \le \tilde{l}_{g,1}^2 \|y_{\lambda}^*(x) - y^*(x)\|^2.
\end{align*}
\fi
\cite{kwon2023fully} has shown that $\|y_{\lambda}^*(x) - y^*(x)\|$ is always bounded by $O(1/\lambda)$, and thus, the variance of $\hat{\grad} \mL_{\lambda}^*(x)$ is bounded by 
$O(\sigma^2 + \tilde{l}_{g,1}^2)$. Hence, the $O(\epsilon^{-4})$ upper bound can be achieved by running {\it e.g.,} standard SGD with smooth stochastic oracles. 

The key challenge in obtaining upper bounds in both cases is to control the bias. As we cannot directly access $y_{\lambda}^*(x)$ and $y^*(x)$, we must compute estimators of them (say $y$ and $z$) by, {\it e.g.,} running additional gradient steps on $y,z$ in the inner loop before updating $x$. Then bias comes from the error $\|y-y_{\lambda}^*(x)\|$ and $\|z - y^*(x)\|$ after the inner-loop iterations, and thus, our analysis requires tight control of overall bias to keep the number of inner-loop iterations optimal.

\subsubsection{Lower Bound: Slower Progress}
Our lower bound builds on the construction of probabilistic zero-chains by Arjevani et al. \cite{arjevani2023lower} for (single-level) stochastic nonconvex optimization. The key idea of zero-chain is to create a function in a way that, when making a query to an oracle, it discloses only up to one new coordinate. 
With stochastic oracles, the required number of iterations can be amplified by constructing an oracle that exposes the next new coordinate only with a small probability $p \ll 1$. The result from \cite{arjevani2023lower} constructs such an oracle with $p=O(\epsilon^2)$, achieving the $\Omega(\epsilon^{-4})$ lower bound for single-level optimization.

In Bilevel optimization, obtaining the right lower-bound dependency on $\epsilon$ requires care, otherwise, we easily end up with vacuous lower bounds no better than the known lower bounds for single-level optimization. In the noiseless setting, recent work in \cite{chen2023near} has achieved $O(\epsilon^{-2})$ upper bound, which is optimal in $\epsilon$ (the known lower bound of $\Omega(\epsilon^{-2})$ in deterministic single-level optimization \cite{carmon2020lower} also applies here). This implies that objective functions of the hard-instance cannot have a zero-chain longer than $O(\epsilon^{-2})$. In turn, with variance-bounded stochastic oracles, the $\Omega(\epsilon^{-6})$ lower bound must result from the slowdown of progression in zero-chains by stochastic noises due to the smaller probability of showing the next coordinate.

Our key observation  is that we can set the probability $p$ of progression much smaller if the progress in $x$ comes indirectly from $g$. Specifically, we design $y^*(x)$ such that $y^*(x) = F(x)$ where $F(x)$ is the hard instance given in \cite{arjevani2023lower} (see \eqref{eq:chain_basic} for the explicit construction), and let $f,g$ such that
\begin{align}
    f(x,y) = y, \quad g(x,y) = (y - F(x))^2. \label{eq:bilevel_hard_instance}
\end{align}
With the oracle model in Definition \ref{def:ystart_aware} with $r=\Theta(\epsilon)$, the probability of progression can be set $p = O(\epsilon^{4})$ over the length $\Omega(\epsilon^{-2})$ zero-chain, and we obtain the desired $\Omega(\epsilon^{-6})$ lower bound for first-order methods with variance-bounded stochastic oracles. The $\Omega(\epsilon^{-4})$ lower bound with additional stochastic smoothness can be shown on the same construction with minor modifications on scaling parameters. 


\subsection{Related Work}
\label{subsec:related_work}
Due to the vast volume of stochastic optimization, we only review the most relevant lines of work. 

\vspace{-0.2cm}
\paragraph{Upper Bounds for Bilevel Optimization.} 
Bilevel optimization has a long history since its introduction in \cite{bracken1973mathematical}. Beyond classical studies on asymptotic landscapes and convergence rates \cite{white1993penalty, vicente1994descent, colson2007overview}, initiated by \cite{ghadimi2018approximation}, there has been a surge of interest in developing iterative optimization methods in large-scale problems for solving \eqref{problem:bilevel}. Most convergence analysis have been performed on the standard setting of unconstrained and strongly-convex lower-level optimization with various assumptions on the oracle access to gradients and Hessians \cite{ghadimi2018approximation, hong2020two, ji2021bilevel, chen2021closing, khanduri2021near, chen2022single, sow2022constrained, dagreou2022framework, ji2021bilevel, kwon2023fully, liu2021value, sow2022constrained, ye2022bome, yang2023achieving}. We mention that there are also a few recent works that study an extension to nonconvex and/or constrained lower-level problems with certain regularity assumptions on the landscape of $g$ around the lower-level solutions, similar to the {\it local} strong-convexity of the lower-level problems \cite{kwon2023penalty, chen2023bilevel, lu2023first, shen2023penalty, xiao2023alternating}. 

\vspace{-0.3cm}

\paragraph{Lower Bounds for Bilevel Optimization.} There are relatively few studies on lower bounds for the Bilevel optimization. While the lower bounds for the single-level stochastic optimization \cite{carmon2020lower, arjevani2023lower} also imply the lower bounds of Bilevel optimization, the lower complexity could be much higher than single-level optimization. To our best knowledge, only the work in \cite{ji2023lower, dagreou2023lower} has studied the fundamental limits of finding $\epsilon$-stationary points of \eqref{problem:bilevel}. However, \cite{ji2023lower} only considers the noiseless setting when the hyperobjective $F(x)$ is convex (note that this is different from assuming $f$ or $g$ is convex). The work in \cite{dagreou2023lower} only considers the case when objective functions are in the form of finite-sum and the second-order oracles are available, and they only provide a lower-bound of single-level finite-sum optimization \cite{zhou2019lower}. 

\vspace{-0.3cm}

\paragraph{Stochastic Nonconvex Optimization.} There exists a long and rich history of stochastic (single-level) optimization for smooth nonconvex functions. With unbiased first-order gradient oracles, it is well-known that vanilla stochastic gradient descent (SGD) converges to a $\epsilon$-stationary point with at most $O(\epsilon^{-4})$ oracle access \cite{ghadimi2013stochastic}, which is shown to be optimal recently by Arjevani et al. \cite{arjevani2023lower}. With smooth stochastic oracles, the rate has been improved to $O(\epsilon^{-3})$ \cite{fang2018spider, cutkosky2019momentum}, which is also shown to be optimal in \cite{arjevani2023lower}. These results are the counterparts to convergence rates of second-order methods for \eqref{problem:bilevel} with unbiased and smooth Jacobian/Hessian stochastic oracles \cite{chen2021closing, khanduri2021near}. The same rate $O(\epsilon^{-3})$ can also be achieved only with first-order oracles, as shown in \cite{yang2023achieving} with assuming higher-order stochastic smoothness.


\section{Preliminaries}
\label{section:prelim}
Throughout the paper, we specify the assumptions on the smoothness of objective functions. The first one is the global smoothness properties: 
\begin{assumption}
    \label{assumption:nice_functions}
    The functions $f$ and $g$ satisfy the following smoothness conditions.
    \vspace{-0.3cm}
    \begin{enumerate}[itemsep=-0.01cm]
        \item[1.] $f,g$ are continuously-differentiable and $l_{f,1}, l_{g,1}$-smooth respectively, jointly in $(x,y)$ over $\mathbb{R}^{d_x\times d_y}$.
        \item[2.] For every $(x,y) \in \mathbb{R}^{d_x \times d_y}$, $\|\grad_y f(x, y)\| \le l_{f,0}$.
        \item[3.] For every $\bar{x} \in \mathbb{R}^{d_x}$, $g(\bar{x}, \cdot)$ is $\mu_g$-strongly convex in $y$.
    \end{enumerate}
\end{assumption}
In addition to the smoothness of individual objective functions, in Bilevel optimization, we also desire the hyperobjective to be smooth. This can be indirectly assumed by the Lipschitzness of derivatives of $g$:
\begin{assumption}
    \label{assumption:hessian_lipschitz_g}
    $\grad^2_{xy} g, \grad^2_{yy} g$ are well-defined and $l_{g,2}$-Lipschitz jointly in $(x,y)$ for all $(x,y) \in \mathbb{R}^{d_x \times d_y}$.
\end{assumption}


\paragraph{Oracle Classes.} We assume that we can access first-order information of objective functions only through stochastic oracles $\mathrm{O}$ that are $y^{*}(x)$-aware for some radius $r = \Omega(\epsilon)$ (see Definition \ref{def:ystart_aware}). Note that if we take $r=\infty$, these oracles become the usual first-order stochastic gradient oracles that can be queried at any $(x,y)$ and $\hat{y}(x)$ is uninformative. Thus our assumption on the oracles includes the conventional ones. We consider that stochastic oracles allow $N$-simultaneous query: the oracle takes $N$-simultaneous query points $(\bm{x}, \bm{y}) = \left\{(x^n, y^n)\right\}_{n=1}^N$ with $N \ge 2$, and the oracle returns $\{(\hat{\grad} f(x^n, y^n; \zeta), \hat{\grad} g(x^n, y^n; \xi), \hat{y}^n(x^n) )\}_{n=1}^N$, where $\hat{y}^n(x^n)$ is an estimator of $y^*(x^n)$. We denote the oracle class 
as $\mO(N, \sigma^2, \tilde{l}_{g,1}, r)$. We note that $\tilde{l}_{g,1} \ge l_{g,1}$ must always hold. 



\paragraph{Additional Notation} Throughout the paper, $\|\cdot\|$ denotes the Euclidean norm for vectors and operator norm for matrices, and  $\texttt{Var}(\cdot)$ denotes the variance of a random vector. We often denote $a \lesssim b$ when the inequality holds up to some absolute constant. $\mathbb{B}(x, r)$ is a Euclidean ball of radius $r > 0$ around a point $x$.

\section{Upper Bounds}
\label{section:upper_bound}
In this section, we prove the $O(\epsilon^{-6})$ and $O(\epsilon^{-4})$ upper bounds without and with stochastic smoothness \eqref{eq:mean_squared_lipschitz}, respectively. We mainly focus on finding a stationary point of the surrogate hyperobjective $\mL_{\lambda}^*(x)$ defined in \eqref{eq:penalty_surrogate}.
Thanks to $\|F(x) - \mL_{\lambda}^*(x)\| = O(\lambda^{-1})$ given in  Lemma 3.1 from \cite{kwon2023fully}, we get $\epsilon$-starionarity of $F(x)$ with a choice of $\lambda = O(\epsilon^{-1})$. We mention here that our upper bounds do not depend on the `reliability radius' of the oracle $r$ as long as $r \ge \frac{6 l_{f,0}}{\mu_g \lambda}$, hence the readers may assume the standard oracle model with $r = \infty$.

\subsection{$O(\epsilon^{-6})$ Upper Bound}
\label{sec:6up}

We show that double-loop \texttt{F$^2$SA} introduced in \cite{kwon2023fully} can be improved by $O(\epsilon)$ from $O(\epsilon^{-7})$ to $O(\epsilon^{-6})$ with suitable choice of outer-loop batch-size $M$ and inner-loop iterations $T$.


As in \cite{kwon2023fully}, the main challenge comes from handling the bias raised every iteration by using approximations of $y_{\lambda}^*(x)$ and $y^*(x)$. 
Specifically, let $y^{k+1}$ and $z^{k+1}$ be the estimates of $y_{\lambda}^*(x^k)$ and $ y^*(x^k)$ at the $k^{th}$ iteration, respectively. Using these estimates we may approximate $\grad \mL_{\lambda}^*(x^k)$ by (see \eqref{eq:grad_penalty}) 
\ifarxiv
\begin{align}
    G_k &:= \grad_x f(x^k, y^{k+1}) + \lambda (\grad_x g(x^k, y^{k+1}) - \grad_x g(x^k, z^{k+1})). \label{eq:biased_expected_gradient}
\end{align}
\else
\begin{align}
\nonumber
    G_k &:= \grad_x f(x^k, y^{k+1})\\ &\quad + \lambda (\grad_x g(x^k, y^{k+1}) - \grad_x g(x^k, z^{k+1})). \label{eq:biased_expected_gradient}
\end{align}
\fi
Comparing $\grad \mL_{\lambda}^*(x^k)$ and $G_{k}$,  it is natural to consider 
\begin{align}
\label{eq:bias1}
    \lambda(\|y^{k+1} - y_{\lambda}^* (x^k)\| + \|z^{k+1} - y^* (x^k)\|)
\end{align} 
as the order of bias of approximating $\grad \mL_{\lambda}^*(x^k)$ by $G_{k}$. For this bias to be less than $\epsilon$, we need $O(\epsilon /\lambda) = O(\epsilon^{2})$ accuracy of $y^{k+1}$ and $z^{k+1}$, which in turn requesting $T \asymp \epsilon^{-4}$ inner-loop iterations (with SGD on strongly-convex functions) before updating $x^k$ with an unbiased estimate of $G_k$. A similar idea has been used in \cite{chen2023near} where $\tilde{O}(\epsilon^{-3})$ bound in the deterministic setting shown in \cite{kwon2023fully} has been improved to $\tilde{O}(\epsilon^{-2})$ with a choice of $T \asymp \log(\lambda)$.

In the outer loop, we have $x^{k+1} = x^k - \alpha \hat{G}_k$ where
\begin{align}
\label{eq:hatg2}
    \hat{G}_k &:= M^{-1} \tssum_{m=1}^{M} h_{x}^{k,m}.
\end{align}
In Algorithm~\ref{algo:penalty_coupling_yz}, we use the following notations:
\ifarxiv
\begin{align*}
    h_{y}^{k,t} &= \lambda^{-1} \hat{\grad}_y f(x^k,y^{k,t}; \zeta^{k,t}) + \hat{\grad}_y g(x^k, y^{k,t}; \xi^{k,t}), \\
    h_{z}^{k,t} &= \hat{\grad}_y g(x^k, z^{k,t}; \xi^{k,t}), \\
    h_{x}^{k,m} &= \hat{\grad}_x f(x^k, y^{k+1}; \zeta_{k,m}^x) + \lambda (\hat{\grad}_xg(x^k, y^{k+1}; \xi_{x}^{k,m}) - \hat{\grad}_x g(x^k, z^{k+1}; \xi^{k,m}_{x}) ).
\end{align*}
\else
\begin{align*}
    h_{y}^{k,t} &= \lambda^{-1} \hat{\grad}_y f(x^k,y^{k,t}; \zeta^{k,t}) + \hat{\grad}_y g(x^k, y^{k,t}; \xi^{k,t}), \\
    h_{z}^{k,t} &= \hat{\grad}_y g(x^k, z^{k,t}; \xi^{k,t}), \\
    h_{x}^{k,m} &= \hat{\grad}_x f(x^k, y^{k+1}; \zeta_{k,m}^x) \ + \\
    &\quad \lambda (\hat{\grad}_xg(x^k, y^{k+1}; \xi_{x}^{k,m}) - \hat{\grad}_x g(x^k, z^{k+1}; \xi^{k,m}_{x}) ).
\end{align*}
\fi
Under \eqref{eq:unbiased_gradients} and \eqref{eq:bounded_variance_gradients},  $\Var(\hat{G}_k):= \E[\|G_k - \hat{G}_k\|^2)] = O(\lambda^2/M)$ as shown in  Lemma~\ref{lemma:Lgrad_variance}. Thus, the variance can be handled with $M \asymp \epsilon^{-4}$, which suffices to get the $O(\epsilon^{-6})$ upper bound. The proof can be found in Appendix~\ref{ap:up1}.


\begin{algorithm}[t]
    \caption{Penalty Methods with Lower-Level Coupling}
    \label{algo:penalty_coupling_yz}
    {{\bf Input:} total outer-loop iterations: $K$, batch size: $M$, step sizes: $\alpha, \{\gamma_t\}_t$, penalty parameter: $\lambda$, Oracle reliability radius: $r$, Trust-region radius for $(y-z)$: $r_\lambda$, initializations: $x_0 \in \mathbb{R}^{d_x}, y_0, z_0$}
    \begin{algorithmic}[1]
        \FOR{$k = 0 ... K-1$}
            \STATE{\color{blue}{\# Initialize Lower-Level Iteration Variables}}
            \STATE{Get $\hat{y}(x^k)$ such that $\|\hat{y}(x^k) - y^*(x^k)\|\le \frac{r}{2}$.}
            \STATE{$y^{k,0} \leftarrow \Pi_{\mathbb{B}(\hat{y}(x^k), \frac{2r}{3})} \{y^k\}$, $\Delta_y = y^{k,0} - y^k$}
            \STATE{{\bf (If $\tilde{l}_{g,1}=\infty$):} $z^{k,0} \leftarrow \Proj{\mathbb{B}(y^{k}, \frac{r}{2})}{z^{k}}$}
            \STATE{{\bf (Else):} $z^{k,0} \leftarrow \Proj{\mathbb{B}(y^{k,0}, r_\lambda)}{z^{k}+\Delta_y}$}
            \STATE{\color{blue}{\# Lower-Level Iteration with Coupling $y,z$}}
            \FOR{$t = 0, ..., T-1$}
                \STATE{$\bar{y}^{k, t} \leftarrow y^{k,t} - \gamma_t h_y^{k,t}$, $\bar{z}^{k,t} \leftarrow z^{k,t} - \gamma_t h_{z}^{k,t}$}
                \STATE{\color{blue}{\# $y^*$-Awareness }}
                \STATE{$y^{k,t+1} \leftarrow \Pi_{\mathbb{B}(\hat{y}(x^k), \frac{2r}{3})} \{\bar{y}^{k,t}\}$, $\Delta_y = y^{k,t+1} - \bar{y}^{k,t}$}
                \STATE{{\bf (If $\tilde{l}_{g,1}=\infty$):} $z^{k,t+1} \leftarrow \Proj{\mathbb{B}(\hat{y}(x^k), \frac{r}{2})}{\bar{z}^{k,t}}$}
                \STATE{{\bf (Else):} $z^{k,t+1} \leftarrow \Proj{\mathbb{B}(y^{k, t+1}, r_\lambda)}{\bar{z}^{k,t}+\Delta_y}$}
            \ENDFOR
            \STATE{$y^{k+1}, z^{k+1} \leftarrow y^{k,T}, z^{k,T}$}
            \STATE{\color{blue}{\# Upper-Level Iteration with Coupling $y,z$}}
            \STATE{$x^{k+1} \leftarrow x^k - \frac{\alpha}{M} \sum_{m=1}^{M} h_{x}^{k,m}$}
        \ENDFOR
    \end{algorithmic}
\end{algorithm}






\begin{theorem}
\label{thm:up1}
    Suppose Assumptions \ref{assumption:nice_functions} and \ref{assumption:hessian_lipschitz_g} hold and let $\lambda = \max\left( \frac{\lambda_0}{\epsilon}, \frac{6 l_{f,0}}{\mu_g r} \right) \asymp \epsilon^{-1}$, $r_{\lambda} = \frac{l_{f,0}}{\mu_g \lambda}$ where $\lambda_0 := \frac{4 l_{f,0} l_{g,1}}{\mu_g^2} \left( l_{f,1} + \frac{2 l_{f,0} l_{g,2}}{\mu_g} \right)$. Under \eqref{eq:unbiased_gradients} and \eqref{eq:bounded_variance_gradients}, Algorithm \ref{algo:penalty_coupling_yz} with $ \alpha \ll \frac{1}{l_{g,1}}$, $ T \asymp \epsilon^{-4}$, and $\gamma_{t}=\left( \frac{2}{\mu_{g}} + \frac{\lambda}{l_{f,1} + \lambda l_{g,1}} \right) \frac{1}{ 1+t}$  finds an $\epsilon$-stationary point of \eqref{problem:bilevel} within $O(\epsilon^{-6})$ iterations.
\end{theorem}
It is worth mentioning that compared to the $O(\epsilon^{-7})$ results in \cite{kwon2023fully}, we improve the upper bound to $O(\epsilon^{-6})$, and do not require additional assumption on the Hessian-Lipschitzness of $f$.

\subsection{$O(\epsilon^{-4})$ Upper Bound}
\label{sec:4up}

If we aim to get $O(\epsilon^{-4})$ upper bound with first-order smooth oracles, $T\asymp\epsilon^{-4}$ inner-loop iterations are too many to achieve the goal. When we have stochastic smoothness ({\it i.e.,} $\tilde{l}_{g,1} < \infty$), we show that choosing $T \asymp \epsilon^{-2}$ is sufficient to obtain the desired convergence rate. The proof can be found in  Appendix~\ref{ap:up2}.

\begin{theorem}
\label{thm:up2}
Suppose that Assumptions \ref{assumption:nice_functions}-\ref{assumption:hessian_lipschitz_g}, \eqref{eq:unbiased_gradients}, \eqref{eq:bounded_variance_gradients}, and \eqref{eq:mean_squared_lipschitz}
hold. Let the algorithm parameters satisfy the following:
    \begin{align}
    \label{eq:rate}
        \lambda \ge \max\left(\frac{\lambda_0}{\epsilon}, \frac{6l_{f,0}}{\mu_g r}\right), \ \ r_{\lambda} = \frac{l_{f,0}}{\mu_g \lambda}, \ \ \alpha \ll 1. 
    \end{align}
Then, Algorithm \ref{algo:penalty_coupling_yz} with $\lambda \asymp \epsilon^{-1}$, $\gamma \asymp \epsilon^2$, $T \asymp \epsilon^{-2}$, $M \asymp \epsilon^{-2}$, and $K \asymp \epsilon^{-2}$ finds an $\epsilon$-stationary point of \eqref{problem:bilevel} within $O(\epsilon^{-4})$ iterations. 
\end{theorem}
To our best knowledge, the best known results under the similar setting achieve the $O(\epsilon^{-5})$ upper bound with the stochastic smoothness of both objective functions $f$ and $g$ jointly in $(x,y)$ \cite{kwon2023fully}. We show that the upper bound can be improved $O(\epsilon^{-4})$ with the only required additional assumption being the stochastic smoothness in $y$.

Our new observation here is that the estimation for the bias can be tightened by using $$v^{k}:= y^{k} - z^k \hbox{ and } v^*(x):= y^*_{\lambda}(x) - y^*(x).$$
The following lemma is the key to tighten the bias in terms of $y$ and $v$:
\begin{lemma}
    \label{lemma:gradL_bias_control0}
    Suppose that Assumptions \ref{assumption:nice_functions}-\ref{assumption:hessian_lipschitz_g}, \eqref{eq:unbiased_gradients}, \eqref{eq:bounded_variance_gradients}, \eqref{eq:mean_squared_lipschitz}, and $r_{\lambda} = \frac{l_{f,0}}{\mu_g \lambda}$
hold. Then,
    \ifarxiv
    \begin{align*}
        \|\grad \mL_{\lambda}^* (x^k) - G_k\| &\le l_{y} \|y^{k+1} - y_{\lambda}^* (x^k)\| + \lambda l_{g,1} \|v^{k+1} - v^*(x^k)\| + \frac{l_{f,0} l_{y}}{\mu_g \lambda},
    \end{align*}
    \else
    \begin{align*}
        \|\grad \mL_{\lambda}^* (x^k) - G_k\| &\le l_{y} \|y^{k+1} - y_{\lambda}^* (x^k)\| \\
        &\quad + \lambda l_{g,1} \|v^{k+1} - v^*(x^k)\| + \frac{l_{f,0} l_{y}}{\mu_g \lambda},
    \end{align*}
    \fi
    where $l_{y}:= l_{f,1} + \frac{l_{g,2} l_{f,0}}{\mu_g}$.
\end{lemma}
The proof is given in Appendix~\ref{ap:up0}. In comparison to \eqref{eq:bias1}, the term $\|y^{k+1} - y_{\lambda}^* (x^k)\|$ on the upper bound does not depend on $\lambda$. As a consequence, $T = O(\epsilon^{-2})$ (with $\gamma = O(\epsilon^{2})$ is enough to obtain $O(\epsilon)$ accuracy of $y^{k+1}$ and $z^{k+1}$.

On the other hand, we observe that the variance of stochastic noises in updating $v^k = y^{k}-z^{k}$ can be bounded by $O(\|y^{k} - z^{k}\|^2)$, which can be bounded by $O(1/\lambda^2)$ (instead of $O(\sigma^2) = O(1)$) with a forced projection step. We observe that the projection makes the distance between $v^k$ and $v^*$ smaller. Let $\bar{v}^{k,t} := \bar{y}^{k,t} - \bar{z}^{k,t}$.

\begin{proposition}    \label{lemma:projection_smaller_v}
    Suppose that Assumptions \ref{assumption:nice_functions}-\ref{assumption:hessian_lipschitz_g}, \eqref{eq:unbiased_gradients}, \eqref{eq:bounded_variance_gradients}, and \eqref{eq:mean_squared_lipschitz}
hold. Then, for all $k,t$ with $r_{\lambda} = \frac{l_{f,0}}{\lambda \mu_g}$: 
\begin{align}
\label{eq:pv}
    \|v^{k,t+1} - v_k^*\| \le \|\bar{v}^{k,t} - v_{k}^*\|.
\end{align}
\end{proposition}
Hence, it suffices to have $O(\epsilon^{-2})$ inner-loop iterations to make the bias in $v^k$ less than $O(\epsilon/\lambda)$, giving $O(\epsilon^{-4})$ upper bound with first-order smooth oracles. 

\begin{proposition}
    \label{prop:final_upper_bound}
    Under Assumptions \ref{assumption:nice_functions}-\ref{assumption:hessian_lipschitz_g}, \eqref{eq:unbiased_gradients}, \eqref{eq:bounded_variance_gradients},  \eqref{eq:mean_squared_lipschitz}, \eqref{eq:rate}, and $\gamma T > 1/\mu_g$,  the iterates of Algorithm \ref{algo:penalty_coupling_yz} satisfy:
    \ifarxiv
    \begin{align}
         \frac{\alpha}{n} \sum_{k=0}^{n-1} \E[ \lVert \nabla \mathcal{L}_{\lambda}^{*}(x^{k}) \rVert^{2} \leq 
    O(n^{-1}) + 
          \frac{\alpha}{n} \sum_{k=0}^{n-1} \Var(\hat{G}_k) +  O(\lambda^{-2}) +  O(\gamma).
    \end{align}
    \else
    \begin{align}
    \nonumber
         &\frac{\alpha}{n} \sum_{k=0}^{n-1} \E[ \lVert \nabla \mathcal{L}_{\lambda}^{*}(x^{k}) \rVert^{2} \\&\leq 
    O(n^{-1}) + 
          \frac{\alpha}{n} \sum_{k=0}^{n-1} \Var(\hat{G}_k) +  O(\lambda^{-2}) +  O(\gamma).
    \end{align}
    \fi
\end{proposition}
Lastly, we obtain the tighter estimation of $\Var(\hat{G}_k) = O(M^{-1})$ shown in Lemma~\ref{lemma:Lgrad_variance}. Choosing 
$T \asymp \epsilon^{-2}$, $M \asymp \epsilon^{-2}$, and $K \asymp \epsilon^{-2}$,
we conclude the theorem.



\section{Lower Bounds}
\label{section:lower_bound}

In our lower bound construction, we mainly focus on the iteration complexity dependence on $\epsilon$, and we consider a subset of the function class in which smoothness parameters are absolute constants: 
\ifarxiv
\begin{align}
    \mathcal{F}(1) := &\{(f,g) \text{ satisfy Assumptions \ref{assumption:nice_functions}, \ref{assumption:hessian_lipschitz_g}} \ | \ l_{f,0}, l_{f,1}, \mu_g, l_{g,1}, l_{g,2} = O(1) \}. \label{eq:const_condition_function_class}
\end{align}
\else
\begin{align}
    \mathcal{F}(1) := &\{(f,g) \text{ satisfy Assumptions \ref{assumption:nice_functions}, \nonumber \ref{assumption:hessian_lipschitz_g}} \ | \\
    &\qquad \ l_{f,0}, l_{f,1}, \mu_g, l_{g,1}, l_{g,2} = O(1) \}. \label{eq:const_condition_function_class}
\end{align}
\fi
Throughout the section, we consider the smoothness parameters as $O(1)$ quantities and assume that $\epsilon > 0$ is sufficiently small such that $\epsilon^{-1}$ dominates any polynomial factors of smoothness parameters. Before we proceed, we describe additional preliminaries for the lower bound.

\paragraph{Algorithm Class.} We consider algorithms that access an unknown pair of functions $(f,g)$ via a first-order stochastic oracle $\mathrm{O}$. Following the black-box optimization model \cite{nemirovskij1983problem}, we consider an algorithm $\mathtt{A} \in \mA$ {\it paired} with $\mathrm{O} \in \mO(N, \sigma^2, \tilde{l}_{g,1}, r_{\epsilon})$ where $\tilde{l}_{g,1} \ge 100$ and $r_\epsilon := 100 \epsilon$. At any iteration step $t \in \mathbb{N}$, $\mathtt{A}$ takes the first $(t-1)$ oracle responses and generates the next (randomized) query points:
\ifarxiv
\begin{align*}
    (\bm{x}^{t}, \bm{y}^{t}) = \mathtt{A} &\big( \xi_{\mathrm{A}}, (\bm{x}^0, \bm{y}^0), \mathrm{O}(\bm{x}^0, \bm{y}^0; \zeta^0, \xi^0), ..., (\bm{x}^{t-1}, \bm{y}^{t-1}), \mathrm{O}(\bm{x}^{t-1}, \bm{y}^{t-1}; \zeta^{t-1}, \xi^{t-1}) \big),
\end{align*}
\else
\begin{align*}
    (\bm{x}^{t}, \bm{y}^{t}) = \mathtt{A} &\big( \xi_{\mathrm{A}}, (\bm{x}^0, \bm{y}^0), \mathrm{O}(\bm{x}^0, \bm{y}^0; \zeta^0, \xi^0), ..., \\
    & (\bm{x}^{t-1}, \bm{y}^{t-1}), \mathrm{O}(\bm{x}^{t-1}, \bm{y}^{t-1}; \zeta^{t-1}, \xi^{t-1}) \big),
\end{align*}
\fi
where $\xi_{\mathrm{A}}$ is a random seed generated at the beginning of the optimization procedure. For simplicity, we assume that $\sigma_f = \sigma_g = \sigma$ in this section.

\paragraph{Zero-Respecting Algorithms.} An important subclass of randomized algorithm class $\mA$ is the zero-respecting algorithm class $\mA_{\textrm{zr}}$ which preserves the support of gradients with respect to $x$ at the queried points:

\begin{definition}
    Suppose $\hat{\nabla}_x f = 0$. An algorithm $\mathtt{A}$ is zero-respecting if for all $t \ge 0$ and $n \in [N]$ generated by $\mathtt{A}$, 
    \begin{align*}
        \supp(x^{t,n}) \subseteq \textstyle &\bigcup_{t'<t, n' \in [N]} \supp \left( \hat{\nabla}_x g(x^{t',n'}, y^{t',n'}; \xi^{t'}) \right).
    \end{align*}
\end{definition}
Zero-respecting algorithm class plays a critical role in dimension-free lower-bound arguments, as we can construct a family of hard instances for all randomized black-box algorithms from one hard example  (see Section \ref{subsec:conversion_to_randomized}).

\paragraph{Probabilistic Zero-Chains.} At the core of the lower-bound construction is the construction of a hard example for all zero-respecting algorithms. To study the iteration complexity of a general randomized algorithm class, \cite{arjevani2023lower} introduced the notion of {\it progress} defined as the following:
\begin{align}
    \label{eq:def_progress}
    \prog_{\alpha}(x) := \max\{ i \ge 0 | |x_i| > \alpha \}, \text{ ($x_0 \equiv 1$)}.
\end{align}
Here, $\alpha \in [0,1)$ controls the effective threshold of values that would be considered as zero. We also adopt the notion of probabilistic zero-chains from \cite{arjevani2023lower}:
\begin{definition}
    A function $g(x, y)$, paired with gradient estimators $\hat{\nabla} g(x,y; \xi)$ with an independent random variable $\xi$, is a probability-$p$ zero-chain in $x$ if there exists $\alpha > 0$ such that for all $(x, y)$:
    \begin{align}
        &\PP(\prog_0 (\hat{\nabla}_x g(x,y; \xi)) > \prog_{\alpha} (x) + 1) = 0, \nonumber \\
        &\PP(\prog_0 (\hat{\nabla}_x g(x,y; \xi)) = \prog_{\alpha} (x) + 1) \le p. \label{eq:prob_zero_chain_cond}
    \end{align}
\end{definition}

\subsection{Hard Instance}
We start with the base construction from \cite{carmon2020lower} for single-level nonconvex optimization:
\begin{align}
    F(x) &= \epsilon^2 \tssum_{i=1}^{d_x} f_i(x), \nonumber \\
    f_i(x) &:= \Psi_{\epsilon}(x_{i-1}) \Phi_{\epsilon}(x_{i}) - \Psi_{\epsilon}(-x_{i-1}) \Phi_{\epsilon}(-x_{i}), \label{eq:chain_basic}
\end{align}
where $d_x = \floor{\epsilon^{-2}}$, $x_0 \equiv \epsilon$, and $\Psi(x), \Phi(x)$ are a scalar function defined as the following:
\begin{align}
    &\Psi(t) := \begin{cases}
        0, & \text{if $t \le 1/2$} \\
        \exp\left(1 - \frac{1}{(2t-1)^2}\right), & \text{otherwise}
    \end{cases}, \nonumber \\
    &\Phi(t) := \sqrt{e} \int^{t}_{-\infty} e^{-\tau^2/2} d\tau \label{eq:def_chain_scaler_functions},
\end{align}
$\Psi_s(t), \Phi_s(t)$ are short-hands of $\Psi(t/s), \Phi(t/s)$. The above construction satisfies the required smoothness and initial value-gap $F(0) - \inf_x F(x) = O(1)$. We set the hyperobjective to be defined as \eqref{eq:chain_basic}, but the progression of an algorithm is slowed down when the $\grad F(x)$ is only indirectly accessed via $\hat{\grad}_x g(x, y; \xi)$ for $\|y - y^*(x)\| \le r_{\epsilon}$.

To set up the Bilevel objectives, we let $f(x,y) =y$ and $g(x,y) = (y - F(x))^2$ as in \eqref{eq:bilevel_hard_instance}. It is straight-forward to see that $y^*(x) = F(x)$ and $f(x, y^*(x)) = F(x)$. Next, we design the expectation of gradient estimators that the oracle returns. To simplify discussion, let $\hat{\grad} f$ are deterministic, {\it i.e.,} $\hat{\grad} f(x,y;\zeta) = \grad f(x,y)$ with probability 1. We define the stochastic gradients to be 
\begin{align}
    \Exs[\hat{\grad}_x g(x,y; \xi)] = \grad_x \left( r_{\epsilon}^2 \cdot \phi\left(\frac{y - F(x)}{r_{\epsilon}}\right)^2 \right),  \label{eq:biased_gradient}
\end{align}
where $\phi(t)$ is a smooth clipping function:
\begin{align}
    \phi(t) := \begin{cases}
        t, & \text{ if } |t| \le 1/2, \\
        t - \frac{1}{e} \int_{1/2}^{t} \Psi(\tau) d\tau, & \text{ else if } t > 1/2 \\    
        t + \frac{1}{e} \int_{1/2}^{-t} \Psi(\tau) d\tau, & \text{ else }
    \end{cases}. \label{eq:smooth_clipping_def}
\end{align}
With the above construction, gradient estimators are unbiased for $|y-y^*(x)| \le r_{\epsilon}$. Let $\grad_x g_b(x,y)$ be the short-hand of \eqref{eq:biased_gradient}.

Next, we design the coordinate-wise gradient estimators of $g$ w.r.t $x$ as the following. Let $\xi \sim \Ber(p)$, and
\begin{align*}
    &\hat{\grad}_{x_i} g(x,y; \xi) = \grad_{x_i} g_b(x,y) \cdot \left(1 + h_i(x) (\xi/p - 1) \right),   
\end{align*}
where $h_i(x)$ is a smooth version of the indicator function $\indic{i > \prog_{\epsilon/4}(x)}$ (see details in Appendix, Lemma \ref{lemma:smooth_indicator}). Finally, we design stochastic gradients w.r.t $y$ and $\hat{y}$:
\begin{align*}
    \hat{\nabla}_y g(x,y; \xi) &= 2 \left(y - \epsilon^2 \cdot \tssum_{i=1}^{d_x} f_i(x) \left(1 + h_i(x) (\xi/p - 1)\right)\right), \\
    \hat{y}(x) &= \epsilon^2 \tssum_{i=1}^{\prog_{\epsilon/2}(x)} f_i(x).
\end{align*}
The next step is to show the lower complexity bounds of ``activating'' the last coordinate of $x$ for zero-respecting algorithms. We note that additional noises in $\hat{\grad}_y g$ and $\hat{y}$ are only required when extending to randomized algorithms, and $\hat{\grad}_y g$, $\hat{y}$ satisfy the following:
\begin{lemma}
    \label{lemma:verify_y_gradient}
    There exists some absolute constant $c > 0$ such that for $x \in \mathbb{R}^{d_x}, y \in \mathbb{R}$:
    \begin{align*}
        &\texttt{Var} \left(\hat{\nabla}_y g(x,y;\xi) \right) \lesssim \frac{\epsilon^4}{p} \lesssim \sigma^2, \\
        &\Exs[\|\hat{\grad}_{yy}^2 g(x,y;\xi)\|^2] \le 2. 
    \end{align*}
    Furthermore, it holds that $\Exs[\hat{\grad}_y g(x,y;\xi)] = \grad_y g(x,y)$ and $|\hat{y}(x) - F(x)| = O(\epsilon^2) \ll r_{\epsilon}$ for all $x \in \mathbb{R}^{d_x}$ and $y \in \mathbb{R}$.
\end{lemma}

\subsection{Lower Bounds for Zero-Respecting Algorithms}
We start by showing that the construction $(f,g)$ with the oracle $\hat{\nabla} g(x,y;\xi)$ forms a probabilistic zero-chain. The intuition behind the probabilistic zero-chain argument is to amplify the required number of iterations to progress toward the next coordinate from $1$ to $O(1/p)$, which gives the overall lower bound of $\Omega\left(d_x / p\right)$. Thus, as we set $p$ smaller, the more iterations all zero-respecting algorithms would need to reach the last coordinate of $x$. The minimum possible value of $p$ is decided by \eqref{eq:bounded_variance_gradients} on bounded variances:
\begin{align*}
    \texttt{Var} \left(\hat{\nabla}_x g(x,y;\xi) \right) \lesssim \frac{\|\grad_x g_b (x,y)\|_{\infty}^2}{p} \lesssim \sigma^2,
\end{align*}
and the smoothness condition \eqref{eq:mean_squared_lipschitz}:
\begin{align*}
    \Exs[\|\hat{\grad}_{xy}^2 g(x,y; \xi) \|^2] \lesssim \|\grad F(x)\|^2 + \frac{\|\grad F(x)\|_{\infty}^2}{p} \lesssim \tilde{l}_{g,1}^2.
\end{align*}
The crucial feature of our construction is the upper bound on $\|\grad_x  g_b (x,y)\|_{\infty}$, which is much smaller than the single-level optimization case:
\begin{lemma}
    \label{lemma:probabilistic_zerochain}
    There exists some absolute constant $c > 0$ such that for all $x \in \mathbb{R}^{d_x}, y \in \mathbb{R}$,
    \begin{align*}
        &\|\grad_x  g_b (x,y)\|_{\infty} \le c r_\epsilon \epsilon = O(\epsilon^2), \\
        &\|\grad F(x)\|_{\infty} \le c \epsilon = O(\epsilon). 
    \end{align*}
    Therefore, \eqref{eq:prob_zero_chain_cond} holds with $p = \max\left(\epsilon^4/\sigma^2, \epsilon^2/\tilde{l}_{g,1}^2\right)$ and $\alpha = \epsilon/4$. 
\end{lemma}
Note that the scenario only with bounded-variance can be explained by setting $\tilde{l}_{g,1} = \infty$. Thus, with $d_x=O(\epsilon^{-2})$, we obtain the lower bound of $\Omega(\epsilon^{-6}), \Omega(\epsilon^{-4})$ for all zero-respecting algorithms without and with the stochastic smoothness assumption, respectively.

\subsection{Lower Bounds for Randomized Algorithms}
\label{subsec:conversion_to_randomized}
To convert the lower bound for all zero-respecting algorithms to randomized algorithm classes, our construction can adopt the ``randomized coordinate-embedding'' argument from \cite{carmon2020lower}. We define a class of hard instances for randomized algorithms:
\begin{align}
    f_U(x,y) &:= y + \frac{1}{10} \|x\|^2, \nonumber \\
    g_U(x,y) &:= (y - F(U^\top \rho(x))^2, \label{eq:final_lower_bound_instance}
\end{align}
where $U$ is a random orthonormal matrix sampled from $\texttt{Ortho}(d, d_x) := \{U \in \mathbb{R}^{d \times d_x} | U^\top U = I\}$ with $d \gg d_x$ (with a slight abuse in notation, the true dimension of $x$ is $d$ instead of $d_x$), and 
\begin{align}
    \rho(x) := x/\sqrt{1 + \|x\|^2 / R^2},
\end{align}
where $R = 250 \epsilon \sqrt{d_x}$. Then we consider the family of hard instances as the following:  
\ifarxiv
\begin{align}
    \mathcal{F}_{\texttt{hard}} := \{(f_U, g_U) | &U \in \texttt{Ortho} (d, d_x),  f_U, g_U \text{ defined in \eqref{eq:final_lower_bound_instance}} \}. \label{eq:final_hard_family}
\end{align}
\else
\begin{align}
    \mathcal{F}_{\texttt{hard}} := \{(f_U, g_U) | &U \in \texttt{Ortho} (d, d_x), \nonumber \\ 
    & f_U, g_U \text{ defined in \eqref{eq:final_lower_bound_instance}} \}. \label{eq:final_hard_family}
\end{align}
\fi
Intuition for the above construction is that for every randomized algorithm $\mathtt{A} \in \mA$, if we select the low-dimensional embedding $U$ uniformly randomly from a unit sphere, the sequence of query points generated by $\mathtt{A}$ behaves as if it is generated by a zero-respecting algorithm in the embedding space. 
The corresponding stochastic gradient estimator is now given by 
\begin{align}
    &\hat{\nabla}_x g_U(x,y; \xi) = J(x)^\top U \hat{\nabla}_x g (U^\top \rho(x), y;\xi), \nonumber \\
    &\hat{\nabla}_y g_U(x,y; \xi) = \hat{\nabla}_y g (U^\top \rho(x), y;\xi), \nonumber \\
    &\hat{y}_U (x) = \hat{y}(U^\top \rho(x)).
    \label{eq:final_stochastic_gradients}
\end{align}
where $J(x)\in \mathbb{R}^{d\times d}$ is a Jacobian of $\rho(x)$. As the remaining steps follow the same arguments in \cite{arjevani2023lower}, we conclude our lower bound for the randomized algorithms with the oracle model in Definition \ref{def:ystart_aware}:
\begin{theorem}
    \label{theorem:final_lower_bound}
    Let $K$ be the minimax oracle complexity for the function class $\mF(1)$ and oracle class $\mO(N, \sigma^2, \tilde{l}_{g,1}^2, r_{\epsilon})$ where $\tilde{l}_{g,1} \ge 100$ and $r_{\epsilon} = 100\epsilon$:
    \begin{align*}
        K := \inf_{\mA} \sup_{\substack{\mathcal{F}(1) \\ \mO(N, \sigma^2, \tilde{l}_{g,1}, r_{\epsilon})}} \inf \left\{K \in \mathbb{N} | \Exs[\|\grad F(x^{K,1}) \|] \le \epsilon \right\}.
    \end{align*}
    Then the minimax oracle complexity must be at least:
    \begin{align*}
        K \gtrsim \min\left( \frac{\sigma^2}{\epsilon^6}, \frac{\tilde{l}_{g,1}^2}{\epsilon^4} \right).
    \end{align*}
\end{theorem}
Lastly, we leave the following conjecture for the lower bound for future investigation. 
\begin{conjecture}[Lower bound with globally unbiased oracles]
    \label{conjecture:lower_bound}
    The current $\Omega(\epsilon^{-6})$ lower bound only holds with $r_{\epsilon} = \Theta(\epsilon)$. For fully general results, the main challenge is to construct a hard-instance where $g(x,y)$ is strongly-convex in $y$ for all $(x,y) \in \mathbb{R}^{d_x\times d_y}$, while at the same time $\|\grad_x g(x,y)\|_{\infty}$ is globally bounded by the same order of $\|\grad_x g(x,y^*(x))\|_{\infty}$ for all $y \in \mathbb{R}^{d_y}$. We conjecture that lower bounds remain the same with standard stochastic oracles ({\it i.e.,} with $r=\infty$).
\end{conjecture}

\section{Concluding Remarks}
In this work, we have studied the complexity of first-order methods for Bilevel optimization with $y^*$-aware oracles. When the oracle gives $r_{\epsilon} = \Theta(\epsilon)$ estimates of $y^*(x)$ along with $O(r_{\epsilon})$-locally reliable gradients, we establish $O(\epsilon^{-6}), O(\epsilon^{-4})$ upper bounds without and with stochastic smoothness, along with the matching $\Omega(\epsilon^{-6}), \Omega(\epsilon^{-4})$ lower bounds. Our upper bound analysis also holds with standard oracles ({\it i.e.,} with $r = \infty$), improving the best-known results given in \cite{kwon2023fully}. The remaining complexity questions include Conjecture \ref{conjecture:lower_bound}, and tight bounds in terms of other smoothness parameters. We conjecture that obtaining such results would require theoretical breakthroughs beyond existing techniques, thereby leaving them as unresolved challenges.

\bibliographystyle{abbrv}
\bibliography{main}

\newpage 

\appendix

\begin{appendices}


\section{First-order analysis for the upper bound} 

\label{ap:up0}

In this section, we establish the upper bounds on the convergence rate claimed in Theorems \ref{thm:up1} and \ref{thm:up2}.
We use the following notations for filtration: for $k=0,1, \cdots, K-1$ and $t = 0,1, \cdots, T-1$
\begin{itemize}
    \item 
    $\mF_k$: the $\sigma$-algebra generated by  $x^i, y^i, z^i, y^{i,t}, z^{i,t}$ for all $i = 0, 1, \cdots, k-1$ and $t = 0, 1, \cdots, T-1$.
    \item 
    $\mF_{k,s}$: the $\sigma$-algebra generated by $\mF_k$ and $y^{k,t}, z^{k,t}$ for all $t = 0, 1, \cdots, s-1$.
    \item 
    $\mF_k':= \mF_{k,T}$.
\end{itemize}


\subsection{Preliminaries}
Following \cite{kwon2023fully}, we reformulate \eqref{problem:bilevel} as the following constrained single-level problem:
\begin{align}
    &\min_{x \in \mX} \quad F(x) := f(x,y^*(x)) \nonumber \quad \textup{s.t.} \quad g(x,y) - g^{*}(x) \le 0, \label{eq:bilevel_constrained} \tag{\textbf{P'}}
\end{align}
where $g^{*}(x):=g(x, y^{*}(x))$. The Lagrangian $\mathcal{L}_{\lambda}$ for \eqref{eq:bilevel_constrained} with multiplier $\lambda>0$ is given as 
\begin{align}
\mathcal{L}_{\lambda}(x,y) = f(x,y) + \lambda(g(x,y)-g^{*}(x)). 
\end{align}
For each $x\in \mathbb{R}^{d_x}$, recall that 
\begin{align}
  y^{*}(x):= \argmin_{y\in \R^{d_{y}}} g(x,y),\quad y^{*}_{\lambda}(x):= \argmin_{y\in \R^{d_{y}}} \mathcal{L}_{\lambda}(x,y),\quad \mathcal{L}_{\lambda}^{*}(x):=\mathcal{L}_{\lambda}(x,y^{*}_{\lambda}(x)).
\end{align}

Algorithm  \ref{algo:penalty_coupling_yz} seeks to optimize $\mathcal{L}_{\lambda}^{*}(x)$,  instead of the hyperobjective $F(x)$ given in \eqref{eq:bilevel_constrained}. By the first-order optimality condition for $y_{\lambda}^{*}(x)$, we have 
\begin{align}\label{eq:1st_order_opt_y_lambda}
 0 = \nabla_{y}  \mathcal{L}_{\lambda}(x,y_{\lambda}^{*}(x)) = \nabla_{y} f(x,y_{\lambda}^{*}(x)) + \lambda \nabla_{y} g(x,y_{\lambda}^{*}(x)). 
\end{align}
According to Lemma 3.1 in \cite{kwon2023fully}, for any $x\in \mathbb{R}^{d_x}$ and $\lambda\ge 2l_{f,1}/\mu_{g}$, 
    \begin{align}
        \nabla \mathcal{L}_{\lambda}^{*}(x)  = \nabla_{x}  f(x, y^{*}_{\lambda}(x)) + \lambda \left( \nabla_{x} g(x,y^{*}_{\lambda}(x)) - \nabla_{x} g(x,y^{*}(x)) \right). 
    \end{align}
Hence, in order to estimate $\nabla \mathcal{L}_{\lambda}^{*}(x)$, one needs to estimate $y_{\lambda}^{*}(x)$ and $y^{*}(x)$. This are achieved by the inner loop so that $y^{k,T}\approx y_{\lambda}^{*}(x^{k})$ and $z^{k,T}\approx y^{*}(x^{k})$.

Here, we establish two preliminary lemmas that will be used in the subsequent sections. Lemma \ref{lemma:gradL_bias_control} below bounds the bias of the expected gradient estimator $G_{k}$ in \eqref{eq:biased_expected_gradient} for the gradient $\grad \mL_{\lambda}^* (x^k)$ of the surrogate hyperobjective. Note that the second part of this lemma is the restatement of Lemma~\ref{lemma:gradL_bias_control0}.

Denote 
\begin{align*}
    v^*(x) := y_\lambda^*(x) - y^*(x) \hbox{ and } v^k = y^k - z^k.   
\end{align*}
Also, we often use a short-hand $y_{\lambda,k}^* = y_{\lambda}^*(x^k), y_k^* = y^*(x^k)$, and $v^*_k = v^*(x^k)$. Recall that $x^{k+1} = x^k - \alpha \hat{G}_k$ where $\hat{G}_k$ is given in \eqref{eq:hatg2}.

\begin{lemma}
\label{lemma:gradL_bias_control}

Suppose that Assumptions \ref{assumption:nice_functions}-\ref{assumption:hessian_lipschitz_g}, \eqref{eq:unbiased_gradients}, and \eqref{eq:bounded_variance_gradients}
hold.
\begin{enumerate}
    \item[(i)]
    Then, we have
\begin{align}
\label{eq:gradient_estimation_error_pf3}
    \|\grad \mL_{\lambda}^* (x^k) - G_k\| \le (l_{f,1} + \lambda l_{g,1})  \left(\|y^{k+1} - y_{\lambda}^* (x^k)\| + \|z^{k+1} - y^* (x^k)\|\right).
\end{align}
\item[(ii)]
If we further assume \eqref{eq:mean_squared_lipschitz} and $r_{\lambda} = \frac{l_{f,0}}{\mu_g \lambda}$, then
    \begin{align*}
        \|\grad \mL_{\lambda}^* (x^k) - G_k\| \leq l_{y} \|y^{k+1} - y_{\lambda}^* (x^k)\| + \lambda l_{g,1} \|v^{k+1} - v^*(x^k)\| + \frac{l_{f,0} l_{y}}{\mu_g \lambda}
    \end{align*}
    where $l_{y}:= l_{f,1} + \frac{l_{g,2} l_{f,0}}{\mu_g}$.
\end{enumerate}

\end{lemma}

\begin{proof}

The first part directly follows from the smoothness properties of $f(x,\cdot)$ and $g(x,\cdot)$. Let us show the second part. By \eqref{eq:mean_squared_lipschitz}, 
    \begin{align*}
        \|\grad \mL^* (x^k) - G_k\| &\le l_{f,1} \|y^{k+1} - y_{\lambda,k}^*\| + \lambda \|\grad_{xy}^2 g(x,y_{\lambda,k}^*) (y_{k}^* - y_{\lambda, k}^*) - \grad_{xy}^2 g(x,y^{k+1}) (z^{k+1} - y^{k+1})\|  \\
        &\qquad + \lambda l_{g,2} \left( \|y_k^*-y_{\lambda,k}^*\|^2 + \|z^{k+1} - y^{k+1}\|^2 \right) \\
        &\le l_{f,1} \|y^{k+1} - y_{\lambda,k}^*\| + \lambda l_{g,2} \|y_{\lambda,k}^* - y^{k+1}\| \|v^*_{k}\| 
        \\
        &\qquad + \lambda l_{g,1} \|v_{k}^* - v^{k+1}\| + \lambda l_{g,2} \left( \|v_k^*\|^2 + \|v^{k+1}\|^2 \right).
    \end{align*}
    In the last inequality, we use
    \begin{align*}
        &\grad_{xy}^2 g(x,y_{\lambda,k}^*) (y_{k}^* - y_{\lambda, k}^*) - \grad_{xy}^2 g(x,y^{k+1}) (z^{k+1} - y^{k+1}) \\
        &\qquad = (\grad_{xy}^2 g(x,y_{\lambda,k}^*) - \grad_{xy}^2 g(x,y^{k+1})) (y_{k}^* - y_{\lambda, k}^*) + \grad_{xy}^2 g(x,y^{k+1})(y_{k}^* - y_{\lambda, k}^* - z^{k+1} - y^{k+1}).
    \end{align*}
    Finally, note that $\|v_{k}^*\|, \|v^{k}\| < r_{\lambda}$ for all $k$ due to the projection step and $r_\lambda \asymp \lambda^{-1}$, and thus
    \begin{align*}
        \|\grad \mL^* (x^k) - G_k\| &\le (l_{f,1} + l_{g,2} \lambda r_{\lambda}) \|y^{k+1} - y_{\lambda,k}^*\| + \lambda l_{g,1} \|v_{k}^* - v^{k+1}\| + l_{g,2} \lambda r_{\lambda}^2. 
    \end{align*}
\end{proof}









Next, the following lemma gives a bound on the variance of the stochastic gradient estimator $\hat{G}_{k}$ in \eqref{eq:hatg2}. 

\begin{lemma}
   \label{lemma:Lgrad_variance}
    Suppose that Assumptions \ref{assumption:nice_functions}-\ref{assumption:hessian_lipschitz_g}, \eqref{eq:unbiased_gradients}, and \eqref{eq:bounded_variance_gradients}
hold.
    Then, the variance of $\hat{G}_k$ is bounded for all $k$:
    \begin{align*}
        \Var(\hat{G}_k) := \Exs[\|\hat{G}_k - G_k\|^2] \le \frac{1}{M} \left(2\sigma_f^2 + 8 \lambda^2 \sigma_g^2\right). 
    \end{align*}
    If we further assume \eqref{eq:mean_squared_lipschitz} and $r_{\lambda} = \frac{l_{f,0}}{\mu_g \lambda}$, 
    \begin{align*}
        \Var(\hat{G}_k) := \Exs[\|\hat{G}_k - G_k\|^2] \le \frac{1}{M} \left(2\sigma_f^2 + \frac{8 \tilde{l}_{g,1}^2 l_{f,0}^2}{\mu_g^2}\right). 
    \end{align*}
\end{lemma}

\begin{proof}
    Under \eqref{eq:unbiased_gradients}, and \eqref{eq:bounded_variance_gradients}, we can bounded the vairance as the following:
    \begin{align*}
        M \Exs[ \|\hat{G}_k - G_k\|^2] &\le 2\underbrace{\Exs[ \|\hat{\grad}_x f(x^k, y^{k+1}; \zeta_k^x) - \grad_x f(x^k, y^{k+1}) \|^2]}_{\le \sigma_f^2} \\
        &\qquad + 4 \lambda^2 \underbrace{\Exs[ \|\hat{\grad}_x g(x^k, y^{k+1}; \zeta_k^x) - \grad_x g(x^k, y^{k+1})\|^2 ]}_{\le \sigma_g^2} \\
        &\qquad + 4 \lambda^2 \underbrace{\Exs[ \|\hat{\grad}_x g(x^k, z^{k+1}; \zeta_k^x) - \grad_x g(x^k, z^{k+1})\|^2 ]}_{\le \sigma_g^2}.
    \end{align*}
    Under \eqref{eq:mean_squared_lipschitz} and $r_{\lambda} = \frac{l_{f,0}}{\mu_g \lambda}$, we can use different inequality:
    \begin{align*}
        M \Exs[ \|\hat{G}_k - G_k\|^2] &\le 2\underbrace{\Exs[ \|\hat{\grad}_x f(x^k, y^{k+1}; \zeta_k^x) - \grad_x f(x^k, y^{k+1}) \|^2]}_{\le \sigma_f^2} \\
        &\qquad + 4 \lambda^2 \underbrace{\Exs[ \|\hat{\grad}_x g(x^k, y^{k+1}; \zeta_k^x) - \hat{\grad}_x g(x^k, z^{k+1}; \zeta_k^x)\|^2 ]}_{\le \tilde{l}_{g,1} \|y^{k+1} - z^{k+1}\|^2 \le \tilde{l}_{g,1} r_\lambda^2 } \\
        &\qquad + 4 \lambda^2 \underbrace{\Exs[ \|\grad_x g(x^k, y^{k+1}) - \grad_x g(x^k, z^{k+1})\|^2 ]}_{_{\le l_{g,1} \|y^{k+1} - z^{k+1}\|^2 \le l_{g,1} r_\lambda^2 }}.
    \end{align*}
    Nothing that $\tilde{l}_{g,1} \ge l_{g,1}$, we have the lemma.
\end{proof}



\subsection{$O(\epsilon^{-6})$ upper bound}
\label{ap:up1}

In this subsection, we establish the $O(\epsilon^{-6})$ upper bound on the complexity of Algorithm \ref{algo:penalty_coupling_yz} as stated in Theorem \ref{thm:up1}. 

We begin with the standard argument that starts with estimating the one-step change in the (surrogate) objective $\mathcal{L}_{\lambda}^{*}(x^{k+1}) - \mathcal{L}_{\lambda}^{*}(x^{k})$. We denote
\begin{align}
\label{eq:gvg}
   G_k &:=\grad_x f(x^k, y^{k+1}) + \lambda (\grad_x g(x^k, y^{k+1}) - \grad_x g(x^k, z^{k+1})), \\
   \Var(\hat{G}_k) &:= \Exs[\|\hat{G}_k - G_k\|^2].
\end{align}



\begin{proposition}
\label{prop:upl}
    Under the step size rule  $\alpha_k \in (0, \frac{1}{2L})$ for a constant $L$ give in Lemma~\ref{lem:llam}, we have
    \begin{align*}
\Exs[\mathcal{L}_{\lambda}^{*}(x^{k+1}) | \mF_k'] - \mathcal{L}_{\lambda}^{*}(x^{k}) &\leq -\frac{\alpha_k}{2} \lVert \nabla \mathcal{L}_{\lambda}^{*}(x^{k}) \rVert^{2} - \frac{\alpha_k}{4} \|\hat{G}_k\|^2 + \alpha_k  \Var(\hat{G}_k) + \alpha_k\| \nabla \mathcal{L}_{\lambda}^{*}(x^{k}) - G_k \|^2.
\end{align*} 
\end{proposition}

\begin{proof}
    The $L$-smoothness of  $\mathcal{L}_{\lambda}^{*}(x)$ given in Lemma~\ref{lem:llam} and $x^{k+1} = x^k - \alpha_k \hat{G}_k$ yield that
\begin{align*}\label{eq:pf_main_descent_lemma1}
        \mathcal{L}_{\lambda}^{*}(x^{k+1}) - \mathcal{L}_{\lambda}^{*}(x^{k}) &\le  \langle \nabla \mathcal{L}_{\lambda}^{*}(x^{k}),\, x^{k+1}-x^{k} \rangle + \frac{L}{2} \lVert x^{k+1} - x^{k} \rVert^{2},\\
        &=  -\alpha_k \langle \nabla \mathcal{L}_{\lambda}^{*}(x^{k}),\, \hat{G}_k \rangle + \frac{\alpha_k^{2} L}{2} \lVert \hat{G}_k \rVert^{2}.\\
        &= -\frac{\alpha_k}{2} \lVert \nabla \mathcal{L}_{\lambda}^{*}(x^{k}) \rVert^{2} + \frac{\alpha_k^2 L - \alpha_k}{2}  \|\hat{G}_k\|^2 +  \frac{\alpha_k}{2}\| \nabla \mathcal{L}_{\lambda}^{*}(x^{k}) - \hat{G}_k\|^2.
    \end{align*}
Rearranging the right-hand side and using our step size rule  $\alpha_k \in (0, \frac{1}{2L})$, 
\begin{align}
    \mathcal{L}_{\lambda}^{*}(x^{k+1}) - \mathcal{L}_{\lambda}^{*}(x^{k}) &\le -\frac{\alpha_k}{2} \lVert \nabla \mathcal{L}_{\lambda}^{*}(x^{k}) \rVert^{2} - \frac{\alpha_k}{4} \|\hat{G}_k\|^2 +  \frac{\alpha_k}{2}\| \nabla \mathcal{L}_{\lambda}^{*}(x^{k}) - \hat{G}_k\|^2.    
\end{align}
By Young's inequality, the last term on the right-hand side is bounded by
\begin{align}
    \alpha_k\| \nabla \mathcal{L}_{\lambda}^{*}(x^{k}) - G_k \|^2 + \alpha_k \|\hat{G}_k - G_k\|^2.
\end{align}
Taking the expectation on both sides, we conclude our claim. 
\end{proof}






    Next, denote 
    \begin{align}
    \mathcal{J}_{k}:=  \lVert y^{k,0} - y_{\lambda}^{*}(x^{k}) \rVert^{2}  + \lVert z^{k,0} - y^{*}(x^{k}) \rVert^{2}.
    \end{align}
    We derive a recursive inequality  for the above quantity to obtain the following bound on the weighted sum of the squared norm of the expected gradients.

%
\begin{proposition}
\label{prop:upeq}
    Under the same setting as in Theorem~\ref{thm:up1}, we have 
    \begin{align}
    \sum_{k=0}^{n-1} \frac{\alpha_{k} }{2} \E[ \lVert \nabla \mathcal{L}_{\lambda}^{*}(x^{k}) \rVert^{2}]  &\le  \E[  \mathcal{L}_{\lambda}^{*}(x^{0}) ] - \inf_{x\in \R^{d_{x}}} \E[  \mathcal{L}_{\lambda}^{*}(x) ] +  \sigma_{x}^{2} \sum_{k=0}^{n-1} \alpha_{k} +(\E[\mathcal{J}_{0}] + A) \sum_{k=0}^{n-1} \alpha_{k} \frac{16C^{3}/\lambda}{\mu_{g}(1+T)},  \label{eq:gradient_estimation_error_pf5}
    \end{align}
    where $A = (4/\mu_g)^2(\lambda^{-2}\sigma_{f}^{2}+\sigma_{g}^{2} )  + (4/\mu_g)^2 4(\lambda/Cl_{g,1})\sigma_{g}^{2}$.
\end{proposition}

\begin{proof}
    From Proposition~\ref{prop:upl}, we get
    \begin{align}\label{eq:O6_bd_ineq1}
    \sum_{k=0}^{n-1} \frac{\alpha_{k} }{2} \E[ \lVert \nabla \mathcal{L}_{\lambda}^{*}(x^{k}) \rVert^{2}  &\le  \E[  \mathcal{L}_{\lambda}^{*}(x^{0}) ] - \inf_{x\in \R^{d_{x}}} \E[  \mathcal{L}_{\lambda}^{*}(x) ]  \\
    &\qquad - \frac{1}{4 }\underbrace{\sum_{k=0}^{n-1} \left( \alpha_{k}^{-1} \E[\lVert x^{k+1}-x^{k} \rVert^{2}] -  2 \alpha_{k} \E[\lVert \hat{G}_k - \nabla \mathcal{L}_{\lambda}^{*}(x^{k})  \rVert^{2}]  \right)}_{(*)}. \nonumber
    \end{align}
    Below, we will show that the sum $(*)$ above is uniformly lower bounded, which is enough to conclude.

    Note that $G_k$ is the estimated gradient $\nabla \mathcal{L}_{\lambda}^{*}(x^{k})$ with the only source of error being approximation errors in $y^{k,T}\approx y_{\lambda}^{*}(x^{k})$ and $z^{k,T}\approx y^{*}(x^{k})$. The second quantity,  $\hat{G}_k$, is the actual estimated gradient we use in Algorithm \ref{algo:penalty_coupling_yz}, with independent random noise $\xi_{k,T}^{x}, \xi_{k,T}^{xy}, \xi_{k,T}^{xz}$ being the additional source of error.

    
    Using \eqref{eq:gradient_estimation_error_pf3} in Lemma~\ref{lemma:gradL_bias_control}, and Lemma~\ref{lemma:Lgrad_variance}, 
    \begin{align}
    \label{eq:ggg}
       \E[\lVert \hat{G}_k - \nabla \mathcal{L}_{\lambda}^{*}(x^{k}) \,|\, \mF_{k}  \rVert^{2}]   \le 2\sigma_{x}^{2} + 4C^{2} \left( \E\left[ \lVert y_{\lambda}^{*}(x^{k}) - y^{k,T} \rVert^{2} \,|\, \mF_{k}  \right] +  \E\left[ \lVert y^{*}(x^{k}) - z^{k,T} \rVert^{2}\,|\, \mF_{k} \right] \right)
    \end{align}
    where $\sigma_x^2:= \frac{\sigma_{f}^{2} + 2\lambda^{2} \sigma_{g}^{2}}{M}$ and $C=l_{f,1} + \lambda l_{g,1}$.
    Recall that $y^{k+1}=y^{k,T}$ is obtained by minimizing $\lambda^{-1} \mathcal{L}_{\lambda}(x,\cdot)$ using a SGD with an unbiased gradient estimator with variance $\lambda^{-2} \sigma_{f}^{2} + \sigma_{g}^{2}=O(1)$ (see Algorithm \ref{algo:penalty_coupling_yz}). The objective $\lambda^{-1} \mathcal{L}_{\lambda}(x,\cdot)$ is $(\mu_{g}/2)$-strongly convex (by Lemma \ref{lem:L_strong_convex}) and $C/\lambda=O(1)$-smooth where $C:= l_{f,1} + \lambda l_{g,1}$. 
    Then, we choose  the  inner-loop step-size 
    \begin{align}
        \gamma_{t} = \left( \frac{2}{\mu_{g}} + \frac{\lambda}{C} \right) \frac{1}{ 1+t}.
    \end{align}
    We then apply Lemma \ref{lem:SGD_sc_complexity} with $\mu=\mu_{g}/2$, $L=C/\lambda=O(1)$, $\beta=\frac{2}{\mu_{g}} + \frac{\lambda}{C}=O(1)$, $\sigma^{2}=\lambda^{-2} \sigma_{f}^{2} + \sigma_{g}^{2}=O(1)$,  $\gamma = 1$ 
    and the constraint set $\mathcal{B}$ to be the radius $2r/3$-ball around $\hat{y}(x)$, which contains $y^{*}(x)$ by the hypothesis. (When $r=\infty$, $\mathcal{B}$ becomes the whole space.) This gives us  \begin{align}\label{eq:y_SGD_bd_pf11}
        \E[\lVert y^{k,T} - y_{\lambda}^{*}(x^{k}) \rVert^{2} \,|\, \mF_{k}] \le \frac{  \max\left\{ 
        \tilde{C},\,  \lVert y^{k,0} - y_{\lambda}^{*}(x^{k}) \rVert^{2} \right\} }{ 1 + T}
    \end{align}
    for some constant $\tilde{C}=O(1)$. 
    
    Similarly, since $g(x,\cdot)$ is $\mu_{g}$-strongly convex and $l_{g,1}$-smooth, 
     \begin{align}\label{eq:z_SGD_bd_pf1}
        \E[\lVert z^{k,T} - y^{*}(x^{k}) \rVert^{2} \,|\, \mF_{k}] \le \frac{  \max\left\{ \bar{C},\,  \lVert z^{k,0} - y^{*}(x^{k}) \rVert^{2} \right\} }{ 1 + T}
    \end{align}
     for some constant $\bar{C}=O(1)$. 
    Then bounding the maximum of nonnegative quantities by their sum and combining \eqref{eq:ggg}, \eqref{eq:y_SGD_bd_pf11}, and \eqref{eq:z_SGD_bd_pf1}, 
    \begin{align}
      & \E\left[ \lVert G_k - \nabla \mathcal{L}_{\lambda}^{*}(x^{k}) \rVert^{2} \,\bigg|\, \mF_{k}  \right]  \le \frac{8C^{3}/\lambda}{\mu_{g}(1+T)} \bigg[ \underbrace{ \tilde{C} + \bar{C} }_{=:A=O(1)} +  \mathcal{J}_{k}  \bigg]. \label{eq:grad_estimation_error1}
    \end{align}

    Next, we derive a recursion for $\E[\mathcal{J}_{k}]$.  By Young's inequality, \eqref{eq:y_SGD_bd_pf11}, and that $y^{*}_{\lambda}$ is $(4l_{g,1}/\mu_{g})$-Lipschitz continuous (see Lemma \ref{lem:hyper_error}), we obtain 
    \begin{align}
       &\E[  \lVert y^{k+1,0} - y_{\lambda}^{*}(x^{k+1}) \rVert^{2} \,|\, \mF_{k}] \\
       & \qquad  \le 2  \E[ \lVert y^{k,T} - y_{\lambda}^{*}(x^{k}) \rVert^{2}\,|\, \mF_{k}] + 2 \E[ \lVert y_{\lambda}^{*}(x^{k+1}) - y_{\lambda}^{*}(x^{k}) \rVert^{2} \,|\, \mF_{k}] \\
        & \qquad  \le \frac{8(C/\lambda) \max\left\{ \tilde{C},\,  \lVert y^{k,0} - y_{\lambda}^{*}(x^{k}) \rVert^{2} \right\} }{ \mu_{g} (1 + T)} + \frac{32 l_{g,1}^{2}}{\mu_{g}^{2}} \E[ \lVert x^{k+1}-x^{k} \rVert^{2} \,|\, \mF_{k}  ].
    \end{align}
    Similarly, we also have 
    \begin{align}
       &\E[  \lVert z^{k+1,0} - y^{*}(x^{k+1}) \rVert^{2} \,|\, \mF_{k}] \\
        & \qquad  \le \frac{4 l_{g,1} \max\left\{ \bar{C},\,  \lVert z^{k,0} - y^{*}(x^{k}) \rVert^{2} \right\} }{ \mu_{g} (1 + T)} + \frac{2 l_{g,1}^{2}}{\mu_{g}^{2}} \E[ \lVert x^{k+1}-x^{k} \rVert^{2} \,|\, \mF_{k}  ].
    \end{align}
    Suppose $T$ is large enough so that $\mu_{g}(1+T) \ge 16C/\lambda$. Then using \eqref{eq:y_SGD_bd_pf11} and \eqref{eq:z_SGD_bd_pf1}, 
    \begin{align}
       \E[\mathcal{J}_{k+1} \,|\, \mF_{k}] \le  \frac{\mathcal{J}_{k}}{2} + \frac{A}{2} +   \frac{34 l_{g,1}^{2}}{\mu_{g}^{2}} \E[ \lVert x^{k+1}-x^{k} \rVert^{2} \,|\, \mF_{k}  ],
    \end{align}
    where $A$ is the constant defined in \eqref{eq:grad_estimation_error1}. Taking the full expectation and by induction, we get 
    \begin{align}\label{eq:gradient_estimation_error_pf4}
        \E[ \mathcal{J}_{k} ]  \le 2^{-k}(\E[\mathcal{J}_{0}] + A) +  \frac{34 l_{g,1}^{2}}{\mu_{g}^{2}} \sum_{i=0}^{k} \left( \frac{1}{2}\right)^{k-1-i} \E[ \lVert x^{i+1}-x^{i} \rVert^{2}  ].
    \end{align}
    It follows that, combining  \eqref{eq:grad_estimation_error1} and \eqref{eq:gradient_estimation_error_pf4}, we get 
    \begin{align}
      \frac{\mu_{g}}{8C^{3}/\lambda}  \sum_{k=0}^{n-1} \alpha_{k} \E\left[ \lVert G_k - \nabla \mathcal{L}_{\lambda}^{*}(x^{k}) \rVert^{2} \right] &\le  \sum_{k=0}^{n-1} \frac{ \alpha_{k}}{(1+T)} \left( A+\E[ \mathcal{J}_{k} ] \right) \\
        &\le (\E[\mathcal{J}_{0}] + A) \sum_{k=0}^{n-1} \frac{2  \alpha_{k}}{(1+T)}   + \frac{68  l_{g,1}^{2}}{\mu_{g}^{2}} \sum_{k=0}^{n-1} \alpha_{k} \lVert x^{k+1}-x^{k} \rVert^{2}.
    \end{align}
   Thus, we deduce for $\alpha_{k}$ sufficiently small,
    \begin{align}
    (*) 
    &\ge -(\E[\mathcal{J}_{0}] + A) \sum_{k=0}^{n-1}  \frac{16\alpha_{k}^{2} C^{3}/\lambda}{\mu_{g}(1+T)} +\sigma_{x}^{2} \sum_{k=0}^{n-1}\alpha_{k}   \\
    &\qquad + \sum_{k=0}^{n-1}  \alpha_{k}^{-1} \E[\lVert x^{k+1}-x^{k} \rVert^{2}] \left( 1 - \frac{4\cdot 8\cdot 68 C^{3} l_{g,1}^{2}/\lambda}{\mu_{g}^{3}} \alpha_{k}^{2} \right) \\
    &\ge  -(\E[\mathcal{J}_{0}] + A) \sum_{k=0}^{n-1}  \frac{16\alpha_{k}^{2}C^{3}/\lambda}{\mu_{g}(1+T)} +\sigma_{x}^{2} \sum_{k=0}^{n-1}\alpha_{k}. 
    \end{align}
    Combining the above with \eqref{eq:O6_bd_ineq1}, we conclude \eqref{eq:gradient_estimation_error_pf5}.
\end{proof}


Now, we prove the upper bound $O(\epsilon^{-6})$. 

\medskip

\noindent\textbf{Proof of Theorem~\ref{thm:up1}: } Since $\lambda= O(\epsilon^{-1})$, by Lemma \ref{lem:hyper_error}, $\lVert \nabla \mathcal{L}_{\lambda}^{*}(x^{k}) \rVert = O(\epsilon)$ implies $\lVert \nabla F(x^{k}) \rVert = O(\epsilon)$. Hence it is enough to show that Algorithm \ref{algo:penalty_coupling_yz} finds an $\epsilon$-stationary point of the surrogate objective $\mathcal{L}_{\lambda}^{*}$ within $O(\epsilon^{-6})$ iterations.

    For $M\asymp \lambda^{4}=\epsilon^{-4}$, we have 
    $\sigma_{x}^{2} = \frac{\sigma_{f}^{2} + 2\lambda^{2} \sigma_{g}^{2}}{M} = O(\epsilon^{2})$, 
    Then, by Proposition~\ref{prop:upeq}, to obtain $\epsilon$-stationary point, it is enough to have (recall that $C=\Theta(\lambda)$, $A=O(1)$, and $\sigma_{x}^{2}=\Theta(\epsilon^{2})$) 
    \begin{align}
    \frac{O(1)}{\sum_{k=0}^{n-1} \alpha_{k}} +   \frac{\sigma_{x}^{2} \sum_{k=0}^{n-1} \alpha_{k}}{\sum_{k=0}^{n-1}\alpha_{k}} + \frac{O(\lambda^{2} \sum_{k=0}^{n-1} \frac{\alpha_{k}}{T})}{\sum_{k=0}^{n-1} \alpha_{k}} \le \epsilon^{2}. 
    \end{align}
    Then the above follows from 
	\begin{align}
	 \frac{O(1)}{\sum_{k=0}^{n-1} \alpha_{k}} +	\frac{O(\lambda^{2})}{T} + O(\sigma_{x}^{2})\le \epsilon^{2}.
	\end{align}
	Now we set $T \asymp \epsilon^{-4}$ and recall that $\lambda=\epsilon^{-1}$ and  $\sigma_{x}^{2}=O(\epsilon^{2})$. Therefore the above inequality holds if we choose $\alpha_{k}\equiv \alpha=\Theta(1)$ and $n\asymp \epsilon^{-2}$. The total complexity is $O(\epsilon^{-2}\cdot (M+T)) = O(\epsilon^{-6})$, as desired. 
 \hfill$\Box$

\subsection{$O(\epsilon^{-4})$ upper bound}
\label{ap:up2}

In this section, we provide the proofs of  Theorem~\ref{thm:up2} and Proposition~\ref{prop:final_upper_bound}. Throughout this section, we assume that Assumptions \ref{assumption:nice_functions}-\ref{assumption:hessian_lipschitz_g}, \eqref{eq:unbiased_gradients}, \eqref{eq:bounded_variance_gradients}, and \eqref{eq:mean_squared_lipschitz}
hold and let $\gamma_t \equiv \gamma$ for all $t$.

As illustrated in Section~\ref{sec:4up}, our new key observation in Lemma~\ref{lemma:gradL_bias_control} is the estimation of $\|\grad \mL_{\lambda}^* (x^k) - G_k\|$ in terms of $\|y^{k+1} - y_{\lambda}^* (x^k)\|$ and $\|v^{k+1} - v^*(x^k)\|$. In the subsequent section, we estimate each term. Let 
\begin{align}
\I_k := \|y^k - y_{\lambda}^*(x^k)\|^2 \hbox{ and } \W_k:= \|v^{k} - v^*(x^k)\|^2,
\end{align}

On the other hand, we have the additional projection step for $z$ in the inner loop of Algorithm~\ref{algo:penalty_coupling_yz}:
$$z^{k,t+1} \leftarrow \Proj{\mathbb{B}(y^{k, t+1}, r_\lambda)}{\bar{z}^{k,t} + \Delta_y},$$
where $\Delta_y = y^{k,t+1} - \bar{y}^{k,t+1}$ and therefore $\bar{y}^{k,t} - \bar{z}^{k,t} = y^{k,t+1} - (\bar{z}^{k,t}+\Delta_y)$.  
In Appendix~\ref{ap:vv}, we verify that with the appropriate choice of $r_\lambda$ the projection step makes $\|v^{k+1} - v^*(x^k)\|$ smaller, which yields the desired result.




\subsubsection{Descent lemma for $y$}
\label{ap:ii}



\begin{proposition}
\label{lem:i1}
For given $\beta \geq 2$, assume that 
\begin{align}
    \label{eq:0ii}
    \lambda > \frac{l_{f,1}}{l_{g,1}}, \ \ 
    \gamma \in \left(0, \frac{1}{4 l_{g,1}}\right) \hbox{ and }
        \left(1 - \frac{\mu_g \gamma}{2}\right)^T < \frac{1}{2 \beta}.
    \end{align}
Then, we have
    \begin{align}
    \label{eq:1ii}
      \E[ \|y^{k+1} - y_{\lambda}^* (x^k)\| ^{2} \,|\, \mathcal{F}_{k}]  \le \frac{1}{2 \beta}  \I_{k}+ \frac{4 \gamma}{\mu_g} \Var(\hat{\nabla}_y L_\lambda) \frac{1}{\lambda^2}
    \end{align}
    where 
    \begin{align}
        \Var(\hat{\nabla}_y L_\lambda):= \sup_{x,y} \Var(\hat{\grad}_y f(x,y; \zeta) + \lambda \hat{\grad}_y g(x, y; \xi)).
    \end{align}
    Furthermore, 
    \begin{align}
    \label{eq:2ii}
        \sum_{k=0}^n \E[ \I_{k} ] \leq 2\I_0 +  \frac{8 \gamma}{\mu_g} \Var(\hat{\nabla}_y L_\lambda) \frac{n}{\lambda^2} + \frac{64 l_{g,1}^{2}}{\mu_{g}^{2}} \sum_{k=0}^{n-1}  \E[ \lVert \hat{G}_k \rVert^{2} ].
    \end{align}
\end{proposition}

\begin{proof}
    Recall from Lemma \ref{lem:L_strong_convex} that $\mathcal{L}_{\lambda}(x,\cdot)$ is $(\lambda \mu_{g}/2)$-strongly convex. In addition, for given $x$, $\mathcal{L}_{\lambda}(x,\cdot)$ is $(l_{f,1} + \lambda l_{g,1})$-smooth in $y$. 
    Then, Applying 
    Lemma~\ref{lem:SGD_sc_complexity} with the objective function $y \mapsto \lambda^{-1}\mathcal{L}_{\lambda}(x,y)$, $\mu = \mu_g/2$, $\sigma^2 = \Var(\hat{\nabla}_y L_\lambda)$, $L = 4 l_{g,1}$ and the constraint set $\mathcal{B}$ to be the radius $2r/3$-ball around $\hat{y}(x)$, which contains $y^{*}(x)$ by the hypothesis, we get
\begin{align}\label{eq:y_SGD_bd_pf1}
        \E[\lVert y^{k,T} - y_{\lambda}^{*}(x^{k}) \rVert^{2} \,|\, \mathcal{F}_{k}] \le 
        \left(1 - \frac{\mu_g \gamma}{2}\right)^T \lVert y_{k,0} - y_{\lambda}^{*}(x_{k}) \rVert^{2} + \frac{4 \gamma}{\lambda^2 \mu_g} \Var(\hat{\nabla}_y L_\lambda).
    \end{align}
    Thanks to our assumptions in \eqref{eq:0ii}, we conclude \eqref{eq:1ii}.

    Next, let us estimate $\I_{k+1}$ in terms of $\I_k$. 
   Young's inequality, the $(4l_{g,1}/\mu_{g})$-Lipschitz continuity of $y^{*}_{\lambda}$ in Lemma \ref{lem:hyper_error}, and \eqref{eq:1ii}
       yield
    \begin{align}
       \E[  \I_{k+1} \,|\, \mathcal{F}_{k}] &   \le 2  \E[ \lVert y^{k+1} - y_{\lambda}^{*}(x^{k}) \rVert^{2}\,|\, \mathcal{F}_{k}] + 2 \E[ \lVert y_{\lambda}^{*}(x^{k+1}) - y_{\lambda}^{*}(x^{k}) \rVert^{2} \,|\, \mathcal{F}_{k}] \\
        & \qquad  \le \frac{1}{\beta}  \I_{k}+ \frac{4 \gamma}{\lambda^2 \mu_g} \Var(\hat{\nabla}_y L_\lambda)+ \frac{32 l_{g,1}^{2} \alpha^2}{\mu_{g}^{2}} \E[ \lVert \hat{G}_k \rVert^{2} \,|\, \mathcal{F}_{k}].
    \end{align}
    Taking the full expectation and by induction, we get \begin{align}
        \E[ \I_{k} ]  &\le \beta^{-k}  \I_{0} +   \sum_{i=0}^{k-1} \beta^{-k+1+i} \left( \frac{32 l_{g,1}^{2} \alpha^2}{\mu_{g}^{2}}\E[ \lVert \hat{G}_i \rVert^{2}] + \frac{4 \gamma}{\lambda^2 \mu_g} \Var(\hat{\nabla}_y L_\lambda) \right),\\
        &\leq \beta^{-k} \I_0 + \frac{1}{1 - 1/\beta} \frac{4 \gamma}{\lambda^2 \mu_g} \Var(\hat{\nabla}_y L_\lambda) + 
        \frac{32 l_{g,1}^{2} \alpha^2}{\mu_{g}^{2}} \sum_{i=0}^{k-1} \beta^{-k+1+i} \E[ \lVert \hat{G}_i \rVert^{2}].
    \end{align}
    As $\beta > 2$, we have $\frac{1}{1 - 1/\beta} \leq 2$ and we conclude \eqref{eq:2ii}.
\end{proof}



\subsubsection{Estimates for $v_k^*$}
\label{sec:v}
\label{ap:vv}

Next, we prove Proposition~\ref{lemma:projection_smaller_v}. Recall that $\bar{z}^{k,t} = z^{k,t} - \gamma_{t} \hat{\nabla}_y g(x^k, z^{k,t}; \xi_{k,t}^{y})$ (before projection to the ball around $y^{k+1}$), and $\bar{v}^{k,t} := \bar{y}^{k,t} - \bar{z}^{k,t}$. Also recall that $\Delta_y = y^{k,t+1} - \bar{y}^{k,t+1}$ and therefore $\bar{v}^{k,t} = y^{k,t+1} - (\bar{z}^{k,t}+\Delta_y)$. Note that for an appropriate choice of $t$ satisfying $\|u\| \leq t \leq \|v\|$, the projection of $v$ to a ball of radius $t$ makes the distance to $u$ smaller.


\begin{lemma}
    \label{lemma:projection_shortened_distance}
    For any $u, v \in \mathbb{R}^{d_y}$ with $\|v\| \ge t$ and $\|u\| \le t$ for some $r > 0$, the following holds:
    \begin{align*}
        \left\| \frac{t v}{\|v\|} - u \right\| \le \|v - u\|.
    \end{align*}
\end{lemma}

    

\noindent\textbf{Proof of Proposition~\ref{lemma:projection_smaller_v}:} 
First of all, if for $\bar{v}^{k,t} = \bar{y}^{k,t} - \bar{z}^{k,t}$, $\|\bar{v}^{k,t}\| \leq r_\lambda$, then the equality holds in \eqref{eq:pv}. Otherwise, suppose that $$\|\bar{v}^{k,t}\| > r_\lambda.$$ Recall from Lemma~\ref{lem:y_bd} that for all $x \in \mathbb{R}^{d_x}$, 
    \begin{align}
    \label{eq:rlambda_bd}
        \|v^*(x)\| = \|y_{\lambda}^*(x) - y^*(x)\|  \le \frac{l_{f,0}}{\lambda \mu_g}=r_\lambda.
    \end{align}
    Applying Lemma~\ref{lemma:projection_shortened_distance} with $r = r_\lambda$, we conclude \eqref{eq:pv}.
\hfill$\Box$


Next, we obtain the contraction of $v_{k+1}^* - v_k^*$.


\begin{lemma}
    \label{lemma:v_star_contraction}
    For all $k$, we have
    \begin{align*}
        \|v^*(x^{k+1}) - v^*(x^{k})\|^2 = \|v_{k+1}^* - v_k^*\|^2 &\lesssim \frac{l_{g,1}^2 l_v^2}{\mu_g^4 \lambda^2} \|x^{k+1} - x^k\|^2 + \frac{l_{g,2}^2 l_{f,0}^4}{\mu_g^6 \lambda^4} 
    \end{align*}
    where $l_v := l_{g,1} + \frac{l_{f,0}l_{g,2}}{ \mu_g}$.
\end{lemma}

\begin{proof}
    We first show that 
    \begin{align}
        v^*(x) = y_{\lambda}^*(x) - y^*(x) = - \frac{1}{\lambda} (\grad_{yy}^2 g(x,y^*(x)))^{-1} \grad_y f(x,y^*(x)) + O\left( \frac{l_{g,2} l_{f,0}^2}{\mu_g^3 \lambda^2} \right). \label{eq:vstar_basic_form}
    \end{align}
    To see this, note that
    \begin{align*}
        \grad_y g(x, y^*(x)) = 0, \quad \grad_y g(x,y^*_{\lambda} (x)) + \lambda^{-1} \grad_y f(x,y^*_{\lambda}(x)) = 0,
    \end{align*}
    and thus, 
    \begin{align*}
        -\frac{1}{\lambda} \grad_y f(x,y_{\lambda}^*(x)) &= \grad_y g(x,y^*_{\lambda} (x)) - \grad_y g(x, y^*(x)) \\
        &= \grad_{yy}^2 g(x,y^* (x)) (y^*_{\lambda}(x) - y^*(x)) + O(l_{g,2} \|y^*(x) - y_{\lambda}^*(x)\|^2) \\
        &= \grad_{yy}^2 g(x,y^* (x)) (y^*_{\lambda}(x) - y^*(x)) + O\left( \frac{l_{g,2} l_{f,0}^2}{\mu_g^3 \lambda^2} \right). 
    \end{align*}
    Multiplying both sides by $\grad^2_{yy} g(x,y^*(x))$ gives \eqref{eq:vstar_basic_form}. Thus, $v_{k+1}^* - v_k^*$ can be expressed as
    \begin{align*}
        &v_{k+1}^* - v_k^* \\&= v^*(x^{k+1}) - v^*(x^k) \\
        &= \frac{1}{\lambda} \left( \grad_{yy}^2 g(x^k, y^*(x^k))^{-1} \grad_y f(x^{k}, y^*(x^k)) - \grad_{yy}^2 g(x^{k+1}, y^*(x^{k+1}))^{-1} \grad_y f(x^{k+1}, y^*(x^{k+1})) \right) + O\left( \frac{l_{g,2} l_{f,0}^2}{\mu_g^3 \lambda^2} \right) \\
        &= \frac{1}{\lambda} \grad_{yy}^2 g(x^k, y^*(x^k))^{-1} \left( \grad_y f(x^{k}, y^*(x^k)) - \grad_y f(x^{k+1}, y^*(x^{k+1})) \right) \\
        &\quad + \frac{1}{\lambda} \left(\grad_{yy}^2 g(x^k, y^*(x^k))^{-1} - \grad_{yy}^2 g(x^{k+1}, y^*(x^{k+1}))^{-1} \right) \grad_y f(x^{k}, y^*(x^k)) + O\left( \frac{l_{g,2} l_{f,0}^2}{\mu_g^3 \lambda^2} \right) \\
        &= \frac{1}{\lambda} \left( \frac{l_{g,1}}{\mu_g} + \frac{l_{g,2}l_{f,0}}{\mu_g^2} \right) \cdot O\left(\|x^k - x^{k+1}\| + \|y^*(x^k) - y^*(x^{k+1})\| \right) + O\left( \frac{l_{g,2} l_{f,0}^2}{\mu_g^3 \lambda^2} \right). 
    \end{align*}
    Finally, using that $\| y^*(x^k) - y^*(x^{k+1}) \| \le \frac{3l_{g,1}}{\mu_g} \|x^k - x^{k+1}\|$, we have
    \begin{align*}
        \|v_{k+1}^* - v_k^*\| \lesssim \frac{l_{g,1}}{\lambda \mu_g^2} \underbrace{\left(l_{g,1} + \frac{l_{g,2} l_{f,0}}{\mu_g} \right)}_{l_v} \|x^k - x^{k+1}\| + \frac{l_{g,2} l_{f,0}^2}{\mu_g^3 \lambda^2}. 
    \end{align*}
    Having squares on both sides gives the lemma. 
\end{proof}


\begin{lemma} 
    \label{lemma:v_descent}
Suppose that
\begin{align}
\label{as:v_descent}
    \gamma  \in \left(0,  \frac{\mu_g}{4\tilde{l}_{g,1}^2} \right).
\end{align}
    For all $k$ and $t=1,2, \cdots, T-1$, we have
    \begin{align}
    \label{eq:v_descent}
        \Exs[\|\bar{v}^{k,t} - v_k^*\|^2 | \mF_{k,t} ] &\le (1-\mu_g \gamma/2) \|v^{k,t} - v_{k}^*\|^2 + O(\gamma^2)  \frac{\tilde{l}_{g,1}^2 l_{f,0}^2}{\lambda^2\mu_g^2}.
    \end{align}
    
\end{lemma}

\begin{proof}
    In this proof, we only consider the projection step $$z^{k,t+1} \leftarrow \Proj{\mathbb{B}(y^{k, t+1}, r_\lambda)}{\bar{z}^{k,t} + \Delta_y}.$$
    Other projections can be handled by a small modification of Lemma~\ref{lem:SGD_sc_complexity}. 
    By direct computation, we have
    \begin{align*}
        &\Exs[\|\bar{v}^{k,t} - v^*_k\|^2 | \mF_{k,t}] \\
        &= \Exs[\|v^*_k - \bar{v}^{k,t}\|^2 + \|v_{k+1} - \bar{v}^{k,t}\|^2 - 2\vdot{v_k^* - v^{k,t}}{\bar{v}^{k,t} - v^{k,t}} | \mF_{k,t}] \\
        &\le \|v^*_k - v^{k,t}\|^2 + \gamma^2 \cdot \Exs[\|\lambda^{-1}\grad_y f(x^k, y^{k}; \zeta_{k,t}^{y}) + \grad_y g(x^k, y^{k,t}; \xi_{k,t}^{y}) - \grad_y g(x^k, z^{k,t}; \xi_{k,t}^y)\|^2 | \mF_{k,t} ]\\
        &\qquad - 2\gamma \vdot{v_k^* - v^{k,t}}{\lambda^{-1}\grad_y f(x^k, y^{k,t}) + \grad_y g(x^k, y^{k,t}) - \grad_y g(x^k, z^{k,t}) } \\
        &\le (1 - \mu_g \gamma) \|v_k^* - v^{k}\|^2 + \frac{\gamma^2}{\lambda^2} (\sigma_f^2 + l_{f,0}^2) + \gamma^2 \tilde{l}_{g,1}^2 \|y^{k,t} - z^{k,t}\|^2. 
    \end{align*}
    where in the last line, we used co-covercivity of strongly-convex functions $\vdot{x-y}{\grad g(x) - \grad g(y)} \ge \mu_g \|x-y\|^2$. Finally, note that
    \begin{align*}
        \|y^{k,t} - z^{k,t}\|^2 \le 2(\|v^{k,t} - v_{k}^*\|^2 + \|v_{k}^*\|^2 ) \leq 2\|v^{k,t} - v_{k}^*\|^2 + 2r_{\lambda}^2.
    \end{align*}
    Using \eqref{as:v_descent}, we conclude \eqref{eq:v_descent}.
\end{proof}


Lastly, we have the desired estimates for $\|v^{k+1} - v^* (x^k)\|$ and $\W_k$.

\begin{proposition} 
\label{prop:ww}
For given $\beta>2$, assume that \eqref{as:v_descent} and 
\begin{align}
\label{eq:0ww}
        \left(1 - \frac{\mu_g \gamma}{2}\right)^T < \frac{1}{2 \beta}.
    \end{align}
Then,
    \begin{align}
    \label{eq:1ww}
      \E[ \|v^{k+1} - v^* (x^k)\|^2 \,|\, \mathcal{F}_{k}]  \le \frac{1}{2\beta} \E[ \W_{k} ]+ O(\gamma) \frac{\tilde{l}_{g,1}^2 l_{f,0}^2}{\lambda^2\mu_g^3}. 
    \end{align}
    Furthermore, 
    \begin{align}
    \label{eq:2ww}
    \frac{\lambda^2}{n} \sum_{k=0}^n \E[ \W_{k} ] &\leq \frac{1}{n}2 \W_0 \lambda^2 + \frac{l_{g,1}^2 l_{v}^2 \alpha^2}{\mu_g^4  } \frac{1}{n}\sum_{k=0}^{n-1}\|\hat{G}_k\|^2 + 
     C \frac{\tilde{l}_{g,1}^2 l_{f,0}^2}{\mu_g^3} \gamma + \frac{l_{g,2}^2 l_{f,0}^4}{\mu_g^6} \frac{1}{\lambda^2}.
    \end{align}
\end{proposition}

\begin{proof}
    Proposition~\ref{lemma:projection_smaller_v} and Lemma~\ref{lemma:v_descent} yield
    \begin{align}
        \Exs[\| v^{k,t+1} - v^* (x^k)\|^2 | \mathcal{F}_{k,t} ] &\le (1-\mu_g \gamma/2) \|v^{k,t} - v^* (x^k)\|^2 + O(\gamma^2) \cdot \frac{\tilde{l}_{g,1}^2 l_{f,0}^2}{\lambda^2\mu_g^2}.
    \end{align}
    for all $t$ and $k$. By iterating this for $t = 0, 1, \cdots, T-1$, we have
        \begin{align}
        \E[ \|v^{k+1} - v^* (x^k)\| ^{2} | \mathcal{F}_{k} ] = \Exs[\|v^{k,T} - v^* (x^k)\|^2 | \mathcal{F}_{k} ] &\le (1-\mu_g \gamma/2)^T \|v^{k,t} - v_{k}^*\|^2 + O(\gamma) \frac{\tilde{l}_{g,1}^2 l_{f,0}^2}{\lambda^2\mu_g^3}.
    \end{align}
    Using the assumption \eqref{eq:0ww}, the above yields the first inequality \eqref{eq:1ww}.

    On the other hand, using Young's inequality,
    $$\W_{k+1} = \|v^{k+1} - v^*(x^{k+1})\|^2 \le  2 \|v^{k+1} - v^*(x^{k})\|^2 + 2 \|v^*(x^{k+1}) - v^*(x^{k})\|^2.$$ 
    Applying \eqref{eq:1ww} and Lemma~\ref{lemma:v_star_contraction}, we obtain that
    \begin{align}
        \E[ \W_{k+1} | \mathcal{F}_{k} ] &\leq 2 \E[ \|v^{k+1} - v^*(x^{k}) \| ^{2} | \mathcal{F}_{k} ] + \frac{l_{g,1}^2 l_{v}^2 \alpha^2}{\mu_g^4 \lambda^2} \|\hat{G}_k\|^2 + \frac{l_{g,2}^2 l_{f,0}^4}{\mu_g^6 \lambda^4},\\
        &\leq 
        \frac{1}{\beta}  \W_{k} + \frac{l_{g,1}^2 l_{v}^2 \alpha^2}{\mu_g^4} \|\hat{G}_k\|^2\frac{1}{\lambda^2} + C \frac{\tilde{l}_{g,1}^2 l_{f,0}^2}{\lambda^2\mu_g^3} \frac{\gamma}{\lambda^2} + \frac{l_{g,2}^2 l_{f,0}^4}{\mu_g^6} \frac{1}{\lambda^4},
    \end{align}
    for some constant $C>0$. The assumption \eqref{eq:0ww} was used in the last inequality.
    
    Using the parallel argument as in Proposition~\ref{lem:i1}, we conclude \eqref{eq:2ww}.
\end{proof}


\subsubsection{Proofs of Theorem~\ref{thm:up2} and Proposition~\ref{prop:final_upper_bound}}

The following statement is the detailed version of Proposition~\ref{prop:final_upper_bound}.


\begin{proposition}
\label{prop:e4}
Suppose that Assumptions \ref{assumption:nice_functions}-\ref{assumption:hessian_lipschitz_g}, \eqref{eq:unbiased_gradients}, \eqref{eq:bounded_variance_gradients}, and \eqref{eq:mean_squared_lipschitz}. In addition, let the algorithm parameters satisfy \eqref{as:v_descent}, \eqref{eq:0ww},  
    \begin{align}
    r_{\lambda} = \frac{l_{f,0}}{\mu_g \lambda}, \quad \alpha \ll 1,  \quad 
    r \in [ 6r_{\lambda}, \infty]
    \end{align}
and
\begin{align}
\label{eq:1e4}
    -\frac{1}{4} +  C_1 \alpha^2 < 0 \hbox{ where } C_1= l_y^2 \frac{64 l_{g,1}^{2}}{\mu_g^2} +  \frac{l_{g,1}^2 l_{v}^2 }{\mu_g^4}
\end{align}
hold true. 
Then, 
\begin{align}
    &\frac{\alpha}{2} \sum_{k=0}^{n-1} \E[ \lVert \nabla \mathcal{L}_{\lambda}^{*}(x^{k}) \rVert^{2} -   \mathcal{L}_{\lambda}^{*}(x^{0}) + \inf_{x\in \R^{d_{x}}} \E[  \mathcal{L}_{\lambda}^{*}(x) ] \leq 
    \alpha l_y^2 \I_0  +  \alpha l_{g,1}^2 \W_0 \lambda^2 + 
        \alpha \sum_{k=0}^{n-1} \Var(\hat{G}_k) + n O(\lambda^{-2}) + n O(\gamma).
\end{align}
    
\end{proposition}

\begin{proof}
Let $\mathcal{B}_k := \Exs[\mathcal{L}_{\lambda}^{*}(x^{k+1}) | \mF_k] - \mathcal{L}_{\lambda}^{*}(x^{k}) + \frac{\alpha}{2} \lVert \nabla \mathcal{L}_{\lambda}^{*}(x^{k}) \rVert^{2}.$ Recall from Proposition~\ref{prop:upl} that
    \begin{align*}
 \mathcal{B}_k &\leq  - \frac{\alpha}{4} \|\hat{G}_k\|^2 + \alpha  \Var(\hat{G}_k)
 + 2 \alpha \left(\frac{l_{f,0} l_{y}}{\mu_g }\right)^2 \frac{1}{\lambda^2} +  2 \alpha l_{y}^2 \|y^{k+1} - y_{\lambda}^* (x^k)\|^2 + 2 \alpha \lambda^2 l_{g,1}^2 \|v^{k+1} - v^*(x^k)\|^2.
\end{align*}
The upper bounds of the last two terms, $\|y^{k+1} - y_{\lambda}^* (x^k)\|^2$ and $\|v^{k+1} - v^*(x^k)\|^2$, are given in  Proposition~\ref{lem:i1}, and Proposition~\ref{prop:ww}, respectively. Using them, we obtain that
\begin{align*}
    \mathcal{B}_k &\leq - \frac{\alpha}{4} \|\hat{G}_k\|^2 
    + \alpha  \Var(\hat{G}_k)
    + 2 \alpha \left(\frac{l_{f,0} l_{y}}{\mu_g }\right)^2 \frac{1}{\lambda^2}\\ 
    &+ \alpha l_y^2\E[ \I_{k} ] \frac{1}{\beta} + \frac{8 \alpha l_y \gamma}{\mu_g} \Var(\hat{\nabla}\mathcal{L}_\lambda)\frac{1}{\lambda^2}\\ &+  \alpha l_{g,1}^2 \E[ 
 \lambda^2 \W_{k} ] \frac{1}{\beta}  + 2 \alpha l_{g,1}^2  \frac{\tilde{l}_{g,1}^2 l_{f,0}^2}{\mu_g^3} O(\gamma) + 2 \alpha l_{g,1}^2 \frac{l_{g,2}^2 l_{f,0}^4}{\mu_g^6} \frac{1}{\lambda^2}
\end{align*}



    Choosing $\beta = 2$, taking the full expectation and telescoping over $k=0,\dots,n-1$,
    \begin{align}
        \sum_{k=0}^{n-1}  \mathcal{B}_k 
        &\leq \alpha l_y^2 \I_0  +  \alpha l_{g,1}^2 \W_0 \lambda^2 + \alpha \left( -\frac{1}{4} +  C_1 \alpha^2   \right) \sum_{k=0}^{n-1} \|\hat{G}_k\|^2+ 
        \alpha \sum_{k=0}^{n-1} \Var(\hat{G}_k)\\ 
        &
 + \left( 2 \alpha \left(\frac{l_{f,0} l_{y}}{\mu_g }\right)^2 \frac{1}{\lambda^2} + \frac{8}{\mu_g}  \frac{\Var(\hat{\nabla}L_\lambda) \gamma}{\lambda^2} + 2 \alpha l_{g,1}^2 \frac{\tilde{l}_{g,1}^2 l_{f,0}^2}{\mu_g^3}\gamma + 2 \alpha l_{g,1}^2 \frac{l_{g,2}^2 l_{f,0}^4}{\mu_g^6} \frac{1}{\lambda^2}
          \right) n
    \end{align}
for $C_1$ given in \eqref{eq:1e4}.
    

    On the other hand,
    \begin{align}
        &\frac{\alpha}{2} \sum_{k=0}^{n-1} \E[ \lVert \nabla \mathcal{L}_{\lambda}^{*}(x^{k}) \rVert^{2} -   \mathcal{L}_{\lambda}^{*}(x^{0}) + \inf_{x\in \R^{d_{x}}} \E[  \mathcal{L}_{\lambda}^{*}(x) ] 
        \leq \sum_{k=0}^{n-1}  \mathcal{B}_k.
    \end{align} 
    Thanks to \eqref{eq:1e4}, and the fact that $\Var(\hat{\nabla}L_\lambda) = O(\lambda^2)$, we conclude.
\end{proof}

\paragraph{Proof of Theorem~\ref{prop:final_upper_bound}:}
    Since $\lambda= O(\epsilon^{-1})$, by Lemma \ref{lem:hyper_error}, $\lVert \nabla \mathcal{L}_{\lambda}^{*}(x_{k}) \rVert = O(\epsilon)$ implies $\lVert \nabla F(x_{k}) \rVert = O(\epsilon)$. Hence it is enough to show that Algorithm \ref{algo:penalty_coupling_yz} finds an $\epsilon$-stationary point of the surrogate objective $\mathcal{L}_{\lambda}^{*}$ within $O(\epsilon^{-4})$ iterations, respectively.

After the projection steps before entering  the inner loop, we have $\|y^{k,0} - z^{k,0}\| \leq r_\lambda$ for all $k\ge0$. From Proposition~\ref{prop:e4}, we have
\begin{align}
    \frac{\alpha}{2} \frac{1}{n} \sum_{k=0}^{n-1} \E[ \lVert \nabla \mathcal{L}_{\lambda}^{*}(x^{k}) \rVert^{2} \leq 
    \frac{C}{n} + 
          \frac{\alpha}{n} \sum_{k=0}^{n-1} \Var(\hat{G}_k) +  O(\lambda^{-2}) +  O(\gamma).
\end{align}
where $C = 
\mathcal{L}_{\lambda}^{*}(x^{0}) - \inf_{x\in \R^{d_{x}}} \E[  \mathcal{L}_{\lambda}^{*}(x) ] + \alpha l_y^2 \I_0  +  \alpha l_{g,1}^2 \W_0 \lambda^2$.


For $M = O(\epsilon^{-2})$, Lemma~\ref{lemma:Lgrad_variance} yields that $\Var(\hat{G}_k) = O(\epsilon^2)$. Choosing $\gamma = O(\epsilon^{2})$ and $n = O(\epsilon^{-2})$, 
\begin{align*}
    \frac{\alpha}{2} \frac{1}{n} \sum_{k=0}^{n-1} \E[ \lVert \nabla \mathcal{L}_{\lambda}^{*}(x^{k}) \rVert^{2} \leq 
    O(\epsilon^2).
\end{align*}
For $T = O(\epsilon^{-2})$, the step size rule is satisfied \eqref{eq:0ii}. Thus, we conclude that the total complexity is $O(n\cdot (M+T)) = O(\epsilon^{-4})$.





    





\section{Proofs for Lower Bounds} 

In this section, we establish Theorem \ref{theorem:final_lower_bound}. 

\subsection{Auxiliary Lemmas for Lower Bounds}
\begin{lemma}[\cite{carmon2020lower}, Lemma 1]
    \label{lemma:scaler_function_auxiliary_carmon}
    Suppose $\Psi(t), \Phi(t)$ is defined as in \eqref{eq:def_chain_scaler_functions}:
    \begin{align*}
        \Psi(t) = \begin{cases}
            0, & t \le 1/2, \\
            \exp\left(1 - \frac{1}{(2t-1)^2}\right), & t > 1/2,
        \end{cases} \qquad     \Phi(t) = \sqrt{e} \int_{-\infty}^{t} e^{-\frac{1}{2} \tau^2} d\tau.
    \end{align*}
    They satisfy the following properties:
    \begin{enumerate}
        \item Both $\Psi$ and $\Phi$ are infinitely differentiable.
        \item For all $t \in \mathbb{R}$,
        \begin{align*}
            &0 \le \Psi(t) \le e, \ 0 \le \Psi'(t) \le \sqrt{54/e}, \ |\Psi''(t)| \le 32.5 \\
            &0 \le \Phi (t) \sqrt{2\pi e}, \ 0 \le \Phi'(t) \le \sqrt{e}, \ |\Phi''(t)| \le 1. 
        \end{align*}
        \item For all $t \ge 1$ and $|u| \le 1$, $\Psi(t) \Phi'(u) \ge 1$. 
    \end{enumerate}
\end{lemma}
\begin{lemma}[\cite{carmon2020lower}, Lemma 1, 2]
\label{lemma:F_large_coordinates}
    If there exists $0 \le i < d_x$ such that $|x_{i}| \ge \epsilon_x$ and $|x_{i+1}| < \epsilon_x$, then $|\grad_i F(x)| > \epsilon$. Furthermore, for all $x \in \mathbb{R}^{d_x}$, $\|\grad F(x)\|_{\infty} \le 23\epsilon$. 
\end{lemma}
\begin{lemma}
    \label{lemma:auxiliary_clipping}
    Suppose $\phi(t)$ is as defined in \eqref{eq:smooth_clipping_def}:
    \begin{align*}
        \phi(t) = \begin{cases}
            t + \frac{1}{e} \int_{1/2}^{-t} \Psi(\tau) d\tau , & t < -1/2, \\
            t, & -1/2 \le t \le 1/2, \\
            t - \frac{1}{e} \int_{1/2}^t \Psi(\tau) d\tau , & t > 1/2.
        \end{cases}
    \end{align*}
    It satisfies the following properties:
    \begin{enumerate}
        \item $\phi(t)$ is infinitely differentiable.
        \item For all $t \in \mathbb{R}$, 
        \begin{align*}
            |\phi(t)| < 2, \ 0 \le \phi'(t) \le 1, \ |\phi''(t)| \le \sqrt{54/e^3}, \ |\phi'''(t)| \le 32.5 / e. 
        \end{align*}
    \end{enumerate}
\end{lemma}
\begin{proof}
    Other properties come immediately from Lemma \ref{lemma:scaler_function_auxiliary_carmon}, we only prove $|\phi(t)| < 2$. To see this, for $t \ge 1/2$, 
    \begin{align*}
        \phi(t) &= \frac{1}{2} + \int_{1/2}^{t} 1 - \frac{1}{e} \Psi(\tau) d\tau \le 1 + \int_{1}^{t} \frac{1}{(2\tau-1)^2} d\tau \le 2. 
    \end{align*}
    where in the first inequality, we used $1 - x \le \exp(-x)$. The same holds for $t \le -1/2$. 
\end{proof}

\begin{lemma}
    \label{lemma:smooth_indicator}
    Let $f_i(x), h_i(x)$ be defined as the following:
    \begin{align}
        f_i(x) &:= \Psi_{\epsilon}(x_{i-1}) \Phi_{\epsilon}(x_{i}) - \Psi_{\epsilon}(-x_{i-1}) \Phi_{\epsilon}(-x_{i}), \nonumber \\
        h_i(x) &:= \Gamma\left( 1 - \left(\tssum_{j=i}^{d_x} \Gamma^2(|x_j|/\epsilon) \right)^{1/2} \right),
        \label{eq:smooth_indicator}
    \end{align}
    where $\Gamma(t)$ is given by
    \begin{align*}
        \Gamma(t) := \frac{\int_{1/4}^{t} \Lambda(\tau) d\tau}{\int_{1/2}^{1/4} \Lambda(\tau) d\tau }, \quad \text{ where } \quad \Lambda(t) = \begin{cases}
            0, & t \le 1/4 \text{ or } t \ge 1/2, \\
            \exp\left( -\frac{1}{100(t-1/4)(1/2-t)} \right), & \text{otherwise}
        \end{cases}.
    \end{align*}
    Then, $h_i(x)$ satisfies $\indic{i > \prog_{\epsilon/4}(x)} \le h_i(x) \le \indic{i > \prog_{\epsilon/2}(x)}$, $O(1)$-Lipschitzness and the following:
    \begin{align*}
        f_i^{(k)}(x) h_i(x) = 0, \qquad \forall i \neq \prog_{\epsilon/2}(x) + 1, k \in \mathbb{N}_+. 
    \end{align*}
\end{lemma}
\begin{proof}
The construction of $h_i(x)$ is brought from \cite{arjevani2023lower} where the boundedness and $O(1)$-Lipschitzness are guaranteed from the proof of Lemma 4 in \cite{arjevani2023lower}. The last property follows from the fact that any $k^{th}$ order derivative of $\Psi(t)$ for $t \le 1/2$ is 0, and thus,
\begin{align*}
    \Psi_{\epsilon}^{(k)} (x_{i-1}) = 0, \qquad \forall i > \prog_{\epsilon/2}(x) + 1, 
\end{align*}
and
\begin{align*}
    h_i(x) \le \indic{i > \prog_{\epsilon/2}(x)} = 0, \qquad \forall i < \prog_{\epsilon/2}(x) + 1. 
\end{align*}
\end{proof}

\begin{lemma}[\cite{arjevani2023lower}, Lemma 15]
    \label{lemma:rho_smoothness}
    For all $x \in \mathbb{R}^d$, $\rho(x)$ is $1$-Lipschitz and $(3/R)$-smooth, that is,
    \begin{align*}
        \|J(x)\| \le 1, \ \|J(x^1) - J(x^2)\| \le \frac{3}{R}\|x^1 - x^2\|, \qquad \forall x^1, x^2 \in \mathbb{R}^d. 
    \end{align*}
\end{lemma}

The following lemma is the key to converting the zero-chain argument to general randomized algorithms \cite{arjevani2023lower}:
\begin{lemma}[\cite{arjevani2023lower}, General Version of Lemma 5]
    \label{lemma:to_rand_algs}
    Let $\mathtt{A} \in \mA$ be a randomized algorithm that accesses functions $f,g: \mathbb{R}^{d_x \times d_y} \rightarrow \mathbb{R}$ through a stochastic oracle, and generates batched queries with norm-bounded $x$, {\it i.e.,} $\|x^{t,n}\| \le R$ for all $t \in \mathbb{N}, n \in [N]$, with some $R > 0$. Let $U$ be a random matrix uniformly distributed on $\texttt{Ortho}(d, d_x)$ with $d \gtrsim \frac{R^2 N d_x}{p} \log\left( \frac{Nd_x^2}{p \delta}\right)$ and $p, \delta > 0$, and let $u_j$ be the $j^{th}$ column of $U$. Additionally, let $\mathrm{O}_U$ be an oracle that takes a batched query $(\bm{x}, \bm{y}):= \{(x^n, y^n)\}_{n=1}^N$, and returns  $G_U(\bm{x}, \bm{y}; \xi)$ where $\xi$ is a random variable and $G_U$ is the oracle response to $(\bm{x}, \bm{y})$ and $\xi$ parameterized by $U$. Suppose the following holds for $G_U$ with probability $1$:
    \begin{align*}
        G_U(\bm{x}, \bm{y}; \xi) = G_{U'} (\bm{x}, \bm{y}; \xi), \quad \forall U'\in \texttt{Ortho}(d, d_x): u'_j = u_j \ \text{ for all } j = 1, 2, ... , \max_{n} \prog_{1/4} (U^\top x^n) + 1.
    \end{align*}
    and with probability at least $1-p$:
    \begin{align*}
        G_U(\bm{x}, \bm{y}; \xi) = G_{U'} (\bm{x}, \bm{y}; \xi), \quad \forall U' \in \texttt{Ortho}(d, d_x): u'_j = u_j \ \text{ for all } j = 1, 2, ... , \max_{n} \prog_{1/4} (U^\top x^n).
    \end{align*}
    When $\mathtt{A}$ is paired with $\mathrm{O}_U$, with probability at least $1 - \delta$, 
    \begin{align*}
        \max_{n\in[N]} \prog_{1/4} \left(U^\top x^{t,n} \right) < d_x, \qquad \forall t < \frac{d_x - \log(1/\delta)}{2 p}.
    \end{align*}
\end{lemma}
\begin{proof}
    The original proof in \cite{arjevani2023lower} uses the progress of returned gradient estimators directly as the overall progress of an algorithm:
    \begin{align*}
        \gamma^{t}_x = \max_{t'\le t,n\in[N]} \left( \prog_{0} (\hat{\grad}_x g(U^\top x^{t',n}, y^{t',n}; \xi^{t'}) \right).
    \end{align*}
    Our definition is a generalization of the above measure to include other signals $\hat{\grad}_y g$ and $\hat{y}$ in the oracle response. Specifically, we define
    \begin{align*}
        \gamma^{t} = \arg\min_{i\in [d_x]}: G_U(\bm{x}^{t'}, \bm{y}^{t'}; \xi^{t'}) = G_{U'} (\bm{x}^{t'}, \bm{y}^{t'}; \xi^{t'}), \quad \forall t' \le t, \ \forall U'\in \texttt{Ortho} (d, d_x): u_{j}' = u_j \text{ for all } j \in [i].
    \end{align*}
    With the above definition, under the event $\mathcal{B}^{t} := \{\max_{n} \prog_{1/4} (U^\top x^{t,n}) \le \gamma^{t-1}\}$ and filtration $\mathcal{G}^{t-1}$:
    \begin{align*}
        \mathcal{G}^{t-1} := \sigma(\xi_{\mathrm{A}}, (\bm{x}^0, \bm{y}^0), G_U(\bm{x}^0, \bm{y}^0; \xi^0), ..., (\bm{x}^{t-1}, \bm{y}^{t-1}), G_U(\bm{x}^{t-1}, \bm{y}^{t-1}; \xi^{t-1})),
    \end{align*}
    we can check that
    \begin{align*}
        \PP(\gamma^{t} - \gamma^{t-1} \notin \{0,1\}, \mathcal{B}^t | \mathcal{G}^{t-1}) = 0, \\
        \PP(\gamma^{t} - \gamma^{t-1} = 1 , \mathcal{B}^t | \mathcal{G}^{t-1}) \le p.
    \end{align*}
    The rest follows the same steps in \cite{arjevani2023lower}, hence we refer the readers to Appendix B.1 in the reference.
\end{proof}

\subsection{Proof of Lemma \ref{lemma:verify_y_gradient}}
We first note that $\Exs[\hat{\grad}_y g(x,y;\xi)] = \grad_y g(x,y)$, and
\begin{align*}
    \grad_y g(x,y) &= - 2 \left( y - \epsilon^2 \tssum_{i=1}^{d_x} f_i(x) \right) = - 2 \left( y - \epsilon^2 \tssum_{i=1}^{\prog_{\epsilon/2}(x)+1} f_i(x) \right).
\end{align*}
Furthermore, from Lemma \ref{lemma:smooth_indicator}, we can observe that
\begin{align*}
    \hat{\grad}_y g(x,y;\xi) - \grad_y g(x,y) &= 2\epsilon^2 \cdot f_{\prog_{\epsilon/2}(x) +1}(x) h_{\prog_{\epsilon/2}(x) +1}(x) (\xi/p - 1).
\end{align*}
Since both $f_i(x)$ and $h_i(x)$ are bounded by $O(1)$ for all $i$ and $x$, we have $\texttt{Var}(\hat{\grad}_y g(x,y;\xi)) \lesssim \epsilon^4/p$. The remaining properties follow straightforwardly from the construction.

\subsection{Proof of Lemma \ref{lemma:probabilistic_zerochain}}
We start with writing down the explicit formula for $\grad_x g_b(x,y)$:
\begin{align*}
    \grad_x g_b(x,y) &= - r_\epsilon \phi\left(\frac{y - F(x)}{r_{\epsilon}}\right) \phi'\left(\frac{y - F(x)}{r_{\epsilon}}\right) \grad F(x), 
\end{align*}
Thus, 
\begin{align*}
    \|\grad_x g_b(x,y)\|_{\infty} \lesssim r_{\epsilon} \|\grad F(x)\|_{\infty} \lesssim r_{\epsilon} \epsilon,
\end{align*}
where we use Lemma \ref{lemma:F_large_coordinates} for the last inequality. 

Now we show that probabilitic zero-chain property. Recall that we define
\begin{align*}
    \hat{\grad}_{x_i} g(x,y;\xi) = \grad_{x_i} g_b(x,y) \cdot (1 + h_i(x) (\xi/p - 1)), 
\end{align*}
and note that from Lemma \ref{lemma:smooth_indicator},
\begin{align*}
    \grad_{x_i} F(x) h_i(x) = O(\epsilon)\cdot \grad_{x_i} (f_{i-1}(x) + f_i(x)) h_i(x) = 0, \quad \forall i\neq \prog_{1/2} (x)+1.
\end{align*}
Therefore, the probabilistic zero-chain property is satisfied, and $\texttt{Var} \left(\hat{\grad}_x g(x,y;\xi)\right) \lesssim \frac{\|\grad_x g_b(x,y)\|_{\infty}^2}{p}$ which must be less than $O(\sigma^2)$. We also check the stochastic smoothness:
\begin{align*}
    \|\hat{\grad}_x g(x,y^1;\xi) - \hat{\grad}_x g(x,y^2;\xi)\| &\lesssim \left( \max_{y} \| \grad_{xy}^2 g_b(x,y) \| + \max_{y} \| \grad_{xy}^2 g_b(x,y) \|_{\infty} (\xi/p) \right) \|y^1 - y^2\| \\
    &\lesssim \left(\max_x \|F(x)\| + \max_x \|F(x)\|_{\infty} (\xi/p) \right) \|y^1 - y^2\| \\
    &\lesssim (1 + \epsilon \xi /p) \|y^1 - y^2\|. 
\end{align*}
Thus, we have $\Exs[ \|\hat{\grad}_x g(x,y^1;\xi) - \hat{\grad}_x g(x,y^2;\xi)\|^2 ] \lesssim \frac{\epsilon^2}{p} \|y^1 - y^2\|^2$, and this must be less than $\tilde{l}_{g,1}^2 \|y^1 - y^2\|^2$. Thus, we conclude that $p = \Omega\left(\max(r_{\epsilon}^2 \epsilon^2 / \sigma^2, \epsilon^2 / \tilde{l}_{g,1}^2)\right)$.

\subsection{Proof of Theorem \ref{theorem:final_lower_bound}}
We first reiterate the initial value-gap and smoothness properties of the modified hard instance: for any $U \in \texttt{Ortho}(d,d_x)$, let $F_U$ be the hyperobjective constructed with $(f_U, g_U)$ defined in \eqref{eq:final_lower_bound_instance}, paired with oracle responses defined in \eqref{eq:final_stochastic_gradients}. Then, 
    \begin{enumerate}
        \item $F_U(0) - F_U^* \le O(1)$.
        \item $f_U,g_U$ satisfies Assumption \ref{assumption:nice_functions}, \ref{assumption:hessian_lipschitz_g} with $O(1)$ smoothness parameters. 
        \item $\|\grad F_U(x)\| > O(\epsilon)$ if $\prog_{\epsilon/4}(U^\top x) < d_x$. 
        \item For all $(x,y) \in \mathbb{R}^{d_x \times d_y}$, 
        \begin{align*}
            \Exs[\|\hat{\grad}_x g_U(x,y;\xi) - \Exs[\hat{\grad}_x g_U(x,y; \xi)] \|^2] \le O(r_{\epsilon} \epsilon)^2 / p.
        \end{align*}
        \item For all $(x,y) \in \mathbb{R}^{d_x \times d_y}$, 
        \begin{align*}
            \Exs[\|\hat{\grad}_x g_U (x,y^1;\xi) - \hat{\grad}_x g_U(x,y^2; \xi)\|^2] \le O \left(1 + \frac{\epsilon^2}{p} \right) \cdot \|y^1 - y^2\|^2.
        \end{align*}
    \end{enumerate}

The first property immediately follows from the fact that $F(0) = F_U(0)$ and $F_U^* \ge F^*$, and the fact that $F^* = O(\epsilon^2 d_x) = O(1)$. Property 2 is trivial given our construction and Lemma \ref{lemma:scaler_function_auxiliary_carmon}. The third property follows from the proof of Lemma 6 in \cite{arjevani2023lower}. For rest two properties, note that
\begin{align*}
    \Exs[\hat{\grad}_x g_U(x,y;\xi)] = J(x)^\top U \hat{\grad}_x g_b(U^\top \rho(x), y;\xi),
\end{align*}
and since $\|J(x)\| = O(1)$ by Lemma \ref{lemma:rho_smoothness}, we can similarly show the bounded-variance and smoothness as in Lemma \ref{lemma:probabilistic_zerochain}. Finally, we can invoke Lemma \ref{lemma:to_rand_algs}, and we get the theorem.

\section{Auxiliary Lemmas}
\label{appendix:upper_bound}


In Lemma \ref{lem:SGD_sc_complexity} below, we obtain convergence rate of the projected gradient descent (PSGD) algorithm with either fixed or diminishing step sizes. These statements are adapted from the analogous statements for unconstrained SGD for strongly convex objectives (Thm. 5.7 in \cite{garrigos2023handbook} for fixed step sizes and Thm. 4.7 in \cite{bottou2010large} for diminishing step sizes).

\begin{lemma}[Convergence rate of PSGD for strongly convex objectives]\label{lem:SGD_sc_complexity}
    Let $f:\R^{p}\rightarrow \R$ be a $L$-smooth and $\mu$-strongly convex function for some $\mu,L>0$. Suppose $f^{*}:=\inf_{x} f(x)>-\infty$ and denote $x^{*}:=\argmin_{x} f(x)$. Furthermore, let $\mathcal{B}$ be a convex set containing $x^*$.  Suppose $G(x,\xi)$ is an unbiased stochastic gradient estimator for $f$, that is, $\E[G(x,\xi)]=\nabla f(x)$ for all $x\in \mathcal{B}$. Further assume that the variance of the gradient estimation error is bounded: $\E[\lVert G(x,\xi) - \nabla f(x)\rVert^{2}]\le \sigma^{2}$ for all $x\in \mathcal{B}$. Let $\rho:=\frac{2\mu L}{\mu+L}$. Consider the following PSGD iterates: 
    \begin{align}
        x_{t+1} \leftarrow \Pi_{\mathcal{B}} \left\{ x_{t} - \alpha_{t} G(x_{t},\xi_{t}) \right\}.
    \end{align}
    Then the following hold:
    \begin{description}
        \item[(i)] (\textit{fixed step size}) Suppose $\alpha_{t}\equiv \alpha < 2L/(\mu+L)$. Fix $T > \frac{4L \log(L/\mu)}{\mu}$. Then for all $0 \le t \le T$, 
        \begin{align}\label{eq:SGD_sc_complexity}
    \E[ \lVert x^{t} - x^{*} \rVert^{2} ] \le (1-\mu \alpha)^t \|x^0 - x^*\|^2 + \frac{\alpha \sigma^2}{\mu}.
    \end{align}
    Taking $\alpha = \frac{8\log T}{\mu T}$, we have $\Exs[\|x^T - x^*\|^2] \le \frac{1}{T^4} \|x^0 - x^*\|^2 + \frac{8 \log T}{\mu^2 T} \sigma^2$.

        \item[(ii)] (\textit{diminishing step size}) Suppose $\alpha_{t}= \frac{\beta}{\gamma + t}$, where $\beta>1/\rho$, $\gamma>0$ are constants.  Then for all $t\ge 0$, 
    \begin{align}\label{eq:SGD_sc_complexity}
         \E[ \lVert x_{t} - x^{*} \rVert^{2} ] \le \frac{\nu}{\gamma+ t}, 
    \end{align}
    where $\nu:= \max\left\{ \frac{\beta^{2}\sigma^{2}L}{2(\beta \rho - 1)},\, \gamma \lVert x_{0}-x^{*} \rVert^{2} \right\}$. In particular, the assertion holds for $\beta=2/\rho=\frac{\mu+L}{\mu L}$ and $\nu=\max\left\{ \frac{(\mu+L)^{2}\sigma^{2}L}{2\mu^{2}L^{2}},\, \gamma \lVert x_{0}-x^{*} \rVert^{2} \right\}$.
    \end{description}

\end{lemma}


\begin{proof}
    Let $\bar{x}^t := x^t - \alpha G(x^t; \xi_t)$. Then 
    \begin{align*}
        \|\bar{x}^{t} - x^*\|^2 &= \|\bar{x}^{t} - x^*\|^2 + \|x^t - x^*\|^2 + 2\vdot{\bar{x}^{t} - x^t}{x^t - x^*} \\
        &= \alpha_t^2 \|G(x^t; \xi_t)\|^2 + \|x^t - x^*\|^2 - 2 \alpha_t \vdot{x^t - x^*}{G(x^t; \xi_t)}.
    \end{align*}
    Let $\mF_{t}=\sigma(\xi_{1},\dots,\xi_{t-1})$ denote the $\sigma$-algebra generated by the random variables $\xi_{1},\dots,\xi_{t-1}$. Then $x_{t}$ is measurable w.r.t. $\mF_{t}$, so taking conditional expectation with respect to $\mF_{t}$ and using the co-coercivity of strongly convex functions, we have
    \begin{align}\label{eq:PSGD_lem_pf1}
        \Exs[\|\bar{x}^{t} - x^*\|^2 \,|\, \mF_{t}] &= \alpha_{t}^2 \sigma^2 + \alpha_t^2 \|\grad f(x^t)\|^2 + \|x^t - x^*\|^2 - 2\alpha_t \underbrace{\vdot{x^t - x^*}{\grad f(x^t)}}_{\ge \frac{\mu L}{\mu+L}\|x^t - x^*\|^2 + \frac{1}{\mu+L}\|\grad f(x^t)\|^2} \\
        &\le \left(1- \frac{2\alpha_{t} \mu L}{\mu+L} \right) \|x^t - x^*\|^2 + \alpha_{t}^2 \sigma^2,
    \end{align}
    given that $\alpha_{t}\le 2/(\mu+L)$. Then by the projection lemma and taking the full expectation, we have
    \begin{align*}
        \Exs[\|x^{t+1} - x^*\|^2 ] \le \Exs[\|\bar{x}^t - x^*\|^2] \le \left(1- \alpha_{t}\rho \right) \E[\|x^t - x^*\|^2] + \alpha_{t}^2 \sigma^2,
    \end{align*}
    where we have denoted $\rho:=\frac{2\mu L}{\mu+L}$.
    
    To derive \textbf{(i)}, suppose $\alpha_{t}\equiv \alpha<2L/(\mu+L)$. Then applying \eqref{eq:PSGD_lem_pf1} recursively, we have
    \begin{align*}
        \Exs[\|x^{t} - x^*\|^2] &\le (1 - \mu \alpha)^t \|x^0 - x^*\|^2 + 2\sigma^2 \sum_{j=0}^t  (1-\alpha \mu)^j \alpha^2 \\
        &\le (1 - \mu \alpha)^t \|x^0 - x^*\|^2 + \frac{\alpha \sigma^2}{\mu}.
    \end{align*}

    Next, we show \textbf{(ii)} by induction. 
    Indeed, \eqref{eq:SGD_sc_complexity} holds for $t=1$ by the choice of $\nu$. Denoting $\hat{t}:=\gamma+t$, by the induction hypothesis and by the choices of step size $\alpha_{t}$ and $\nu$, 
    \begin{align}
          \E[ \lVert x_{t+1} -x^{*} \rVert^{2}] &\le \left( 1-\frac{\beta\rho}{\hat{t}}\right) \frac{\nu}{\hat{t}} + \frac{L\sigma^{2}\beta^{2}}{2\hat{t}^{2}} \\
          &\le \left( \frac{\hat{t}-1}{\hat{t}^{2}} \right) \nu \underbrace{-  \left( \frac{\beta\rho-1}{\hat{t}^{2}} \right) \nu + \frac{L\sigma^{2}\beta^{2}}{2\hat{t}^{2}}}_{\le 0} \\
          &\le  \frac{\hat{t}-1}{(\hat{t}-1)(\hat{t}+1)}  \nu \le  \frac{\nu}{\gamma+t}.
    \end{align}
    This shows the assertion. 
\end{proof}

\begin{lemma}
\label{lem:llam}
    $\mL_{\lambda}^*(x)$ is $L$-smooth with $L:= \frac{6 l_{g,1}}{\mu_g} \left(l_{f,1} + \frac{l_{g,1}^2}{\mu_g} + \frac{l_{f,0} l_{g,1} l_{g,2}}{\mu_g^2} \right)$.
\end{lemma}
\begin{proof}
    See Lemma B.8 in \cite{chen2023near}. 
\end{proof}



\begin{lemma}\label{lem:L_strong_convex}
    For $\lambda\ge 2l_{f,1}/\mu_{g}$, $\mathcal{L}_{\lambda}(x,\cdot)$ is $(\lambda \mu_{g}/2)$-strongly convex for each $x \in \mathbb{R}^{d_x}$.
\end{lemma}

\begin{proof}
    Since $f(x,\cdot)$ is $l_{f,1}$-smooth and $g(x,\cdot)$ is $\mu_{g}$-strongly convex, 
    \begin{align}
        \mathcal{L}_{\lambda}(x,y') - \mathcal{L}_{\lambda}(x,y) \ge \langle \nabla_{y} \mathcal{L}_{\lambda}(x,y),\, y'-y \rangle + \frac{\lambda \mu_{g}-l_{f,1}}{2}\lVert y'-y \rVert^{2}.
    \end{align}
    Hence if $\lambda \mu_{g} \ge 2 l_{f,1}$, then $\mathcal{L}_{\lambda}(x,\cdot)$  is $(\lambda \mu_{g}/2)$-strongly convex. 
\end{proof}

\begin{lemma}\label{lem:y_bd}
    For $\lambda\ge 2l_{f,1}/\mu_{g}$, it holds that for each $x\in \mathbb{R}^{d_x}$, 
    \begin{align}
        \lVert y^{*}_{\lambda}(x) - y^{*}(x) \rVert \le \frac{2 l_{f,0}}{\lambda \mu_{g}} = O(\lambda^{-1}). 
    \end{align}
\end{lemma}

\begin{proof}
    Note that if a function $w\mapsto h(w)$ is $\mu$-strongly convex and is minimized at $w^{*}$, then for any $w$, 
    \begin{align}
        0 \ge h(w^{*}) - h(w) \ge \langle \nabla h(w),\, w^{*}-w \rangle + \frac{\mu}{2} \lVert w^{*}-w\rVert^{2}. 
    \end{align}
    Using Cauchy-Schwarz inequality, the above inequality implies 
      \begin{align}
            \lVert w^{*} - w \rVert \le \frac{2}{\mu} \lVert \nabla  h(w) \rVert. 
      \end{align}
    Apply the above inequality for $h(y)=g(x,y)$ and the first-order optimality condition for $y_{\lambda}^{*}(x)$ in \eqref{eq:1st_order_opt_y_lambda}
    to get 
  \begin{align}
        \lVert  y^{*}(x) - y^{*}_{\lambda}(x) \rVert \le \frac{2}{\mu_{g}} \lVert \nabla_{y} g(x,y^{*}_{\lambda}(x)) \rVert = \frac{2}{\lambda \mu_{g} } \lVert \nabla_{y} f(x,y^{*}_{\lambda}(x)) \rVert \le \frac{2 l_{f,0}}{\lambda \mu_{g}},
  \end{align}
  as desired. 
\end{proof}

\begin{lemma}\label{lem:hyper_error}
    For $\lambda\ge 2l_{f,1}/\mu_{g}$, it holds that 
    \begin{align}
        |\mathcal{L}_{\lambda}^{*}(x) - F(x)| \le D_{0} \lambda^{-1} = O(\lambda^{-1}),
    \end{align}
where $D_{0}:=\left( l_{f,1} + \frac{l_{f,1}^{2}}{\mu_{g}}   \right) \frac{l_{f,1}}{\mu_{g}}$.
\end{lemma}

\begin{proof}
    See Lemma B.3 in \cite{chen2023near}. 
\end{proof}

\end{appendices}

\end{document}